\documentclass[a4paper,12pt]{scrartcl}

\usepackage[utf8]{inputenc}

\usepackage{amsthm,amsmath,amsfonts,amssymb}
\numberwithin{equation}{section}  
\usepackage{mathrsfs}
\usepackage{todonotes}

\usepackage{tikz-cd}

\usepackage{mathtools}	 
\usepackage{bm}	  
\usepackage{bbm}	 

\usepackage{tensor}	 
\usepackage[colorlinks=true]{hyperref} 
\usepackage[all]{hypcap}       
 
\newcommand{\C}{\mathbb{C}}
\newcommand{\R}{\mathbb{R}}
\newcommand{\Q}{\mathbb{Q}}
\newcommand{\Z}{\mathbb{Z}}
\newcommand{\ii}{\operatorname{i}}

\newcommand{\kk}{\mathfrak{k}}

\newcommand{\ch}{\operatorname{ch}}

\newcommand{\Rea}{\operatorname{Re}}
\newcommand{\Imm}{\operatorname{Im}}

\newcommand{\del}{\partial}

\newcommand{\PP}{\mathbb{P}}

\newcommand{\cY}{\mathcal{Y}}
\newcommand{\cL}{\mathcal{L}}
\newcommand{\cM}{\mathcal{M}}

\newcommand{\cS}{\mathcal{S}}

\newcommand{\cH}{\mathcal{H}}
\newcommand{\cT}{\mathcal{T}}
\newcommand{\cU}{\mathcal{U}}
\newcommand{\olo}{\mathcal{O}}

\newcommand{\Hom}{\operatorname{Hom}}
\newcommand{\End}{\operatorname{End}}
\newcommand{\Bl}{\operatorname{Bl}}

\newcommand{\GM}{\operatorname{GM}}

\newcommand{\Crit}{\operatorname{Crit}}

\newcommand{\pd}{\operatorname{PD}}
\newcommand{\FS}{D\!\operatorname{FS}}

\newcommand{\Gcl}{\widehat{\Gamma}}
\newcommand{\syz}{\operatorname{SYZ}}
\newcommand{\dhym}{\operatorname{dHYM}}
\newcommand{\codim}{\operatorname{codim}}
\newcommand{\Coh}{\operatorname{Coh}}
 
\newtheorem{thm}{Theorem}[section]
 
\newtheorem{prop}[thm]{Proposition}

\newtheorem{cor}[thm]{Corollary}
\newtheorem{conj}[thm]{Conjecture}

\theoremstyle{definition}
\newtheorem{definition}[thm]{Definition}

\theoremstyle{remark}
\newtheorem{exm}[thm]{Example}
\newtheorem{rmk}[thm]{Remark}

\title{Nakai-Moishezon criteria and the toric Thomas-Yau conjecture}
\author{Jacopo Stoppa}
 
\date{\today}

\begin{document}

\maketitle

\begin{abstract} We consider a class of Lagrangian sections $\cL$ contained in certain Calabi-Yau Lagrangian fibrations (mirrors of toric weak Fano manifolds). We prove that a form of the Thomas-Yau conjecture holds in this case: $\cL$ is Hamiltonian isotopic to a special Lagrangian section in this class if  and only if a stability condition holds, in the sense of a slope inequality on objects in a set of exact triangles in the Fukaya-Seidel category. This agrees with general proposals by Li.  

We use the SYZ transform, the toric $\widehat{\Gamma}$-theorem, and toric homological mirror symmetry in order to reduce the statement to one about supercritical deformed Hermitian Yang-Mills connections, known as the Nakai-Moishezon criterion. 

As an application, we prove that, on the mirror of a toric weak del Pezzo surface, if $\cL$ defines a Bridgeland stable object in the Fukaya-Seidel category in a natural sense, then it is Hamiltonian isotopic to a special Lagrangian section in the class. The converse also holds for the mirror of $\Bl_p \PP^2$ and $\Bl_{p, q} \PP^2$, and always holds assuming a conjecture of Arcara and Miles. When $\cL$ is Bridgeland unstable, we obtain a morphism from $\cL$ to a weak solution of the special Lagrangian equation with phase angle satisfying a minimality condition. These results are consistent with general conjectures due to Joyce.

We discuss some generalisations, including a weaker analogue of our main result for general projective toric manifolds, and a similar obstruction, related to Lagrangian multi-sections, in a special case.
\end{abstract}

\section{Introduction}
\subsection{Toric special Lagrangian sections and stability} 
The well-known conjectures of Thomas and Yau \cite{Thomas_MomentMirror, ThomasYau} relate the existence of special Lagrangian submanifolds in Calabi-Yau manifolds to appropriate stability conditions. These conjectures have been updated and expanded in the works of Joyce \cite{Joyce_ThomasYau} and Li \cite{YangLi_ThomasYau}, which also contain many additional references. 

Here we study a version of the Thomas-Yau conjecture in a toric case, namely, for a class of Lagrangian sections of certain Calabi-Yau Lagrangian fibrations, known as toric Landau-Ginzburg models of weak Fano type. 

This is the setup considered by Collins and Yau in \cite{CollinsYau_dHYM_momentmap}, Section 9, which provides important motivation for our work. 

Our other main motivation comes from the results of Li \cite{YangLi_ThomasYau}  (Sections 3.3 - 3.5) showing that, on a Stein almost Calabi-Yau manifold, at least under certain technical assumptions, if $L$, $L_1$, $L_2$ are compact, almost calibrated Lagrangians, with $L$ special Lagrangian, and
\begin{equation*}
L_1 \to L \to L_2 \to L_1[1]
\end{equation*}
is an exact triangle in the relevant Fukaya category, then we have
\begin{equation}\label{LiPhaseInequ}
\arg \int_{L_1} \Omega \leq \arg \int_L \Omega \leq \arg \int_{L_2} \Omega. 
\end{equation}
As Li emphasises (\cite{YangLi_ThomasYau} Sections 1 and 3.5), this is a numerical (cohomological) condition, and here we focus on such numerical aspects.

Let $X$ be a compact, $n$-dimensional K\"ahler manifold. We say that $X$ is weak Fano if $-K_X$ is big and nef. 

Suppose $X$ is a toric weak Fano manifold. Following Givental \cite{Givental_toric}, the \emph{mirror toric Landau-Ginzburg model} corresponding to $X$ is a trivial fibration $\cY \to \cM$ by algebraic tori $\cY_q \cong (\C^*)^n$, endowed with a regular function
\begin{equation*}
W\!: \cY \to \C, 
\end{equation*}
known as the Landau-Ginzburg (LG) potential. Recall that, at least when $X$ is Fano, $W$ can be computed explicitly from a toric fan for $X$ (see Remark \ref{GiventalFormula}). A (possibly complexified) K\"ahler class $[\omega_0]$ defines a point $q \in \cM$, and we denote by 
\begin{equation*}
W(\omega_0) := W|_{\cY_q} 
\end{equation*}
the corresponding LG potential on a fibre, $W(\omega_0)\!: \cY_q \to \C$. In the following, we write 
\begin{equation*}
\Omega_0 = \frac{d x_1}{x_1} \wedge \cdots \wedge \frac{d x_n}{x_n}
\end{equation*}
for the standard holomorphic volume form on the algebraic torus $\cY_q \cong (\C^*)^n$ in a fixed trivialisation of the mirror family $\cY \to \cM$.
We will provide some background on this correspondence in Section \ref{MirrorSec}, following \cite{CoatesCortiIritani_hodge, Iritani_survey}. 

Let $X$ be a toric weak Fano manifold, endowed with a fixed torus-invariant K\"ahler form $ \omega_0$. Let $L \to X$ be a holomorphic line bundle on $X$, endowed with a torus-invariant Hermitian metric $h$. By the classical work of Leung, Yau and Zaslow \cite{LeungYauZaslow} (generalised to toric compact K\"ahler manifolds, as recalled e.g. in \cite{Chan_survey} and in \cite{CollinsYau_dHYM_momentmap}, Section 9), one can construct a corresponding non-compact Lagrangian submanifold of a fixed bounded domain $\cU_q$ in $(\C^*)^n$,
\begin{equation*}
\cL = \cL(L, h) \subset \cU_q \subset \cY_q \cong (\C^*)^n,
\end{equation*} 
with respect to a mirror symplectic form $\omega^{\vee}$ on $\cU_q $.

This is a section of a natural Lagrangian torus fibration $\cY_q \supset \cU_q \to \Delta^0$ corresponding to $\omega_0$, where $\Delta^0$ denotes the interior of the momentum polytope. The Lagrangian section $\cL(L, h)$ is known as the Strominger-Yau-Zaslow (SYZ) transform of $(L, h)$. Under this correspondence, changes in the metric $h$ preserve the Hamiltonian isotopy class of $\cL$. 

Moreover, the K\"ahler form $\omega_0$ on $X$ corresponds to a distinguished holomorphic volume form $\Omega(\omega_0)$ on $\cU_q$. We say that $\cL$ is special Lagrangian if, for some constant phase $e^{-\ii \phi}$, we have 
\begin{equation*}
\Imm\left(e^{-\ii \phi} \Omega(\omega_0)|_{\cL}\right) = 0 
\end{equation*}
\begin{rmk}
When $n > 2$, we will work with special Lagrangians satisfying a crucial assumption, known as \emph{supercritical phase}, which we recall in Section \ref{SupercritSec}, under which the special Lagrangian equation is known to be better behaved (these are solutions for which the lifted Lagrangian phase is sufficiently large).
\end{rmk}
It is known that $\cL(L, h)$ defines an object in a category of Lagrangian submanifolds of $\cY_q$, the Fukaya-Seidel category, denoted by $\FS(\cY_q, W( \omega_0))$, and that the correspondence between $L$ and the class of $\cL(L, h)$ is compatible with \emph{homological mirror symmetry}, i.e. with the equivalence of categories   
\begin{equation}\label{HMSIntro}
D^b(X) \cong \FS(\cY_q, W(\omega_0)),
\end{equation}
where $D^b(X)$ denotes the bounded derived category of coherent sheaves on $X$. 

Indeed, homological mirror symmetry for toric manifolds was studied in detail by many authors, including \cite{Abouzaid_toricHMS, Zaslow_toricHMS} (see e.g. \cite{CollinsYau_dHYM_momentmap}, Section 9 for several additional references). Here we follow the version of \eqref{HMSIntro} due to Fang-Liu-Treumann-Zaslow \cite{Zaslow_toricHMS} and its non-equivariant case (there are various proofs, see e.g. \cite{KuwagakiCohCon, SibillaCohCon, ZhouCohCon}). 

As we will recall, following Fang \cite{Fang_charges}, there are also natural integration cycles $[\tilde{\cL}]$ associated with objects $\tilde{\cL}$ of $\FS(\cY_q, W( \omega_0))$, lying in the rapid decay homology group $H_n(\cY_q, \{\Rea(W( \omega_0)) \gg 0\}; \Z)$, such that, when $\cL = \cL(L, h)$ is a SYZ Lagrangian section and the dual $L^{\vee}$ is ample, then $[\cL]$ is represented by $\cL$ itself. 

Our main result shows that a SYZ Lagrangian section satisfies a form of the Thomas-Yau conjecture. It involves a large scale factor $k > 0$ for both the K\"ahler form and the Lagrangian. We write $q_k \in \cM$ for the point corresponding to $k[\omega_0]$, with fibre $\cY_{q_k} \cong (\C^*)^n$. 
\begin{thm}[Corollary \ref{FSCor}]\label{MainThm} Let $X$ be a toric weak Fano manifold with a fixed torus-invariant K\"ahler form $\omega_0$ and a holomorphic line bundle $L$, with ample inverse $L^{\vee}$. Fix $k > 0$ sufficiently large. Let $\cL = \cL(L^{\otimes k}, h) \subset \cU_{q_k}$ be a SYZ Lagrangian section in the mirror to $(X, k\omega_0)$. Suppose $[\omega_0]$ is \emph{generic} in the sense that it does not lie in the union of finitely many proper analytic subvarieties of $H^{1, 1}(X, \R)$ determined by $c_1(L)$.

Then, as $V$ ranges through toric subvarieties of $X$, there exist objects
\begin{equation*}
\cL_V \in \FS(\cY_{q_k}, W(k \omega_0)),
\end{equation*}
with morphisms
\begin{equation*}
\cL_V[-\codim V] \to \cL,
\end{equation*}
such that $\cL$ is Hamiltonian isotopic to a \emph{supercritical special} SYZ Lagrangian section if, and only if, the phase inequalities for periods
\begin{equation}\label{MainThmPhaseIneq}
\arg\left((-1)^{\codim V}\int_{[\cL_V]} e^{-W(k\omega_0)}\Omega_0 \right) < \arg \int_{ \cL } e^{-W(k\omega_0)}\Omega_0 
\end{equation}
hold.

When $n = 2$, i.e. if $X$ is a toric weak del Pezzo surface, the supercritical assumption is not necessary.   
\end{thm}

\begin{rmk} The role of the scale factor $k > 0$ is discussed in Section \ref{BackSec}, see in particular Remarks \ref{FanoThmRmk}, \ref{CollinsYauRmk}. The holomorphic volume form $e^{-W(k\omega_0)}\Omega_0$ appearing in Theorem \ref{MainThm} is natural in our context, see Remark \ref{VolumeFormRmk}.
\end{rmk}

\begin{rmk} The genericity condition on $[\omega_0]$ is explicit, see Remark \ref{FanoThmRmk} $(v)$.
\end{rmk}
\begin{rmk}\label{UnamplenessRmk} We could work with the opposite convention that $L^{\vee}$ is negative. Then supercritical Lagrangians are replaced by ``conjugate supercritical" ones.

It is possible to remove the assumption that $L^{\vee}$ is ample, at the cost of replacing the Lagrangian $\cL$ as an integration cycle in the phase inequalities for periods \eqref{MainThmPhaseIneq} with the natural associated cycle $[\cL] \in H_n(\cY_{q_k}, \{\Rea(W(k\omega_0)) \gg 0\}; \Z)$. However, in general, if $L^{\vee}$ is not ample, $\cL$ might not define an element in $H_n(\cY_{q_k}, \{\Rea(W(k\omega_0)) \gg 0\}; \Z)$, or, if it does, this might not be cohomologous to the natural $[\cL]$. Thus, it is not clear how \eqref{MainThmPhaseIneq} in this case would make contact with Li's phase inequalities for periods, and so with the Thomas-Yau conjectures.
\end{rmk}

Theorem \ref{MainThm} is stated in a more detailed way as Corollary \ref{ThomasYauCor} in Section \ref{BackSec}. The main tools used in the proof are the SYZ transform, Iritani's toric $\widehat{\Gamma}$-theorem, and toric homological mirror symmetry in the sense of Fang-Liu-Treumann-Zaslow. In this way, Theorem \ref{MainThm} is reduced to a result concerning deformed Hermitian Yang-Mills (dHYM) connections, known as a Nakai-Moishezon criterion. This reduction, together with the necessary background, is explained in detail in Section \ref{BackSec}, see in particular Theorem \ref{FanoThm} and Corollary \ref{FSCor}. Sections \ref{PositivitySec} and \ref{MirrorSec} contain some more technical details about dHYM connections and toric mirror symmetry.

\subsection{Relation to Bridgeland stability on surfaces}
In the case of compact Calabi-Yau manifolds, Joyce \cite{Joyce_ThomasYau} proposed precise conjectures relating the existence of special Lagrangians to stability conditions on the Fukaya category in the sense of Bridgeland. 

We study a toric version of such a relation when $X$ is a toric weak del Pezzo surface. 

Let $E \in \Coh(X)$. Set 
\begin{equation*}
Z_{[\omega_0]}(E) = -\int_X e^{-\ii [\omega_0]}\ch(E).
\end{equation*}
Recall that there exists a natural Bridgeland stability condition $\sigma_{[\omega_0]}$ on $D^b(X)$, with heart given by a suitable tilting $\Coh^{\sharp}(X)$ of $\Coh(X)$ at a torsion pair, and with central charge $Z_{[\omega_0]}(E)$. We will provide some more details in Section \ref{BSec}. See \cite{Collins_stability}, Section 4 for a detailed exposition. 

Firstly, we give an ad-hoc definition of Bridgeland stability for our Lagrangian sections, which nevertheless seems quite natural, as we explain.
\begin{definition}\label{BstabDef} Let $X$ be a toric weak del Pezzo surface. Suppose the fixed line bundle $L^{\vee}$ is positive. Fix $k > 0$ sufficiently large. Let $\cL = \cL(L^{\otimes k}, h) \subset \cU_{q_k}$ be a SYZ Lagrangian section in the mirror to $(X, k\omega_0)$. We say that the Lagrangian 
\begin{equation*}
\tilde{\cL} := \cL[1]
\end{equation*}
is Bridgeland stable iff its mirror $L^{\otimes k}[1]$ is Bridgeland stable with respect to the stability condition $\sigma_{k[\omega_0]} = (\Coh^{\sharp}(X), Z_{k[\omega_0]})$. 
\end{definition}
\begin{rmk} Definition \ref{BstabDef} is consistent. Taking the shift $\tilde{\cL} = \cL[1]$ in the category $\FS(\cY_{q_k}, W(k \omega_0))$ is represented by the same Lagrangian section $\cL$, endowed with the new grading obtained by shifting the lifted Lagrangian phase of $\cL$ by $\pi$, see \cite{YangLi_ThomasYau}, footnote 1 and Remark 2.6. On the other hand, by the cohomological characterisation of the heart $\Coh^{\sharp}(X)$ recalled in Section \ref{BSec}, the shifted line bundle $L^{\otimes k}[1]$ defines an object of $\Coh^{\sharp}(X)$.\end{rmk}
\begin{rmk} Definition \ref{BstabDef} is motivated by the fact that the equivalence \eqref{HMSIntro} is compatible with the SYZ transform for line bundles. 

Note that one could give the same definition on a general toric surface. The key point is that, when $X$ is weak del Pezzo, we will see, relying essentially on the $\widehat{\Gamma}$-theorem, that this definition is natural at the level of central charges for $k \gg 0$ (see Remark \ref{AsympZRmk}). 

That is, we will see that the central charge $Z_{k[\omega_0]}(L^{\otimes k}[1])$ equals the Landau-Ginzburg central charge, given by periods of the holomorphic volume form, up to a small error term,
\begin{equation*}
\int_X e^{- \ii k [\omega_0]} \ch(L^{\otimes k}) = \frac{1}{(2\pi \ii)^{n}} \int_{[\cL]} e^{-W(k \omega_0)/z} \Omega_0\,(1 + O(k^{-1})).
\end{equation*} 
At least when $L^{\vee}$ is ample, the integration cycle $[\cL] \in H_n(\cY_{q_k}, \{\Rea(W(k\omega_0)) \gg 0\}; \Z)$ can be represented by $\cL$ itself (see Remark \ref{UnamplenessRmk}), and this central charge is compatible with Li's phase inequalities for periods \eqref{LiPhaseInequ} and with Joyce's Conjecture 3.2 in \cite{Joyce_ThomasYau}.  
\end{rmk}
\begin{thm}\label{BstabThm} Suppose we are in the situation of Theorem \ref{MainThm} and Definition \ref{BstabDef}, with $X$ is a toric weak del Pezzo surface. Then, if $\tilde{\cL}$ is Bridgeland stable, it is Hamiltionian isotopic to a (shifted) special SYZ Lagrangian section. The converse also holds, conditionally on a conjecture of Arcara-Miles (see Conjecture \ref{AMConj}), which is known when $X = \Bl_p \PP^2$ or $X = \Bl_{p, q} \PP^2$.  
\end{thm}
The proof is given in Section \ref{BstabProof}.
\begin{rmk} The proof of Theorem \ref{BstabThm} is based on showing that, under the assumptions of Theorem \ref{MainThm}, when $X$ is a toric surface (not necessarily weak del Pezzo), Bridgeland stability of $L^{\otimes k}[1]$ with respect to $\sigma_{k[\omega_0]}$ for $k \gg 1$ implies the existence of a dHYM connection on $L^{\otimes k}$, i.e. a solution of \eqref{dHYMIntro}, see Proposition \ref{BstabProp}. Conversely, using an argument of Collins and Shi \cite{Collins_stability}, we will see that the dHYM positivity of $L^{\otimes k}$ with respect to $k[\omega_0]$ implies the Bridgeland stability of $L^{\otimes k}[1]$ with respect to $\sigma_{k[\omega_0]}$ for all $k \geq 1$, if the conjecture of Arcara-Miles holds. The conjecture was proved for $X = \Bl_p \PP^2$ in \cite{ArcaraMiles} and recently for $X = \Bl_{p, q} \PP^2$ in \cite{MizunoYoshida}. 

We note that Theorem \ref{BstabThm} is compatible with the wall-and-chamber decompositions for dHYM connections studied by Khalid and Sj\"ostr\"om Dyrefelt \cite{SohaibZak_higherdim}. 
\end{rmk}
\begin{rmk}\label{CollinsShiExampleRmk} Collins and Shi \cite{Collins_stability} show an example of a K\"ahler class $[\omega_0]$ and a line bundle $L$ on $X = \Bl_p \PP^2$ which is stable with respect to stability condition $(\Coh^{\sharp}(X), Z_{[\omega_0]})$, but for which the dHYM equation \eqref{dHYMIntro} is not solvable, implying that $\tilde{\cL}$ is not Hamiltonian isotopic to a (shifted) special SYZ Lagrangian section. Their $L^{\vee}$ is negative, and so does not satisfy the assumptions of Definition \ref{BstabDef} and Theorem \ref{BstabThm}, although probably the example can be modified.

At any rate we will see that, also in this case, $L^{\otimes k}[1]$ is \emph{unstable} with respect to $(\Coh^{\sharp}(X), Z_{k[\omega_0]})$ for $k \gg 1$, see Example \ref{CollinsShiExm}.
\end{rmk}
\subsubsection{The unstable case}
The works of Joyce \cite{Joyce_ThomasYau} and Li \cite{YangLi_ThomasYau} also contain precise expectations concerning the case of objects $\cL$ in the Fukaya category which are unstable (i.e. not semistable) with respect to the (conjectural) Bridgeland stability condition relevant for the existence of special Lagrangians. These relate to versions of the Harder-Narashiman filtration for $\cL$.   

In our very special situation, when $X$ is a toric weak del Pezzo surface, we can make some progress on the unstable case by relying on the work of Datar-Mete-Song \cite{DatarSong_slopes} concerning weak solutions of the dHYM equation of surfaces. These are briefly recalled in Section \ref{UnstableSec}.

Let $X$ be a toric weak del Pezzo surface. Suppose the fixed ample line bundle $L^{\vee}$ on $X$ does \emph{not} support a smooth solution of the dHYM equation \eqref{dHYMIntro}, so that $\cL = \cL(L, h) \subset \cU_q$ cannot be Hamiltinian isotopic to a smooth special SYZ Lagrangian section. In this case, following the results of \cite{DatarSong_slopes}, there exists a \emph{minimal slope $\R$-divisor $D$} in $X$, which attains the \emph{minimal angle} $\theta_{min} \in (0,\pi)$ determined by
\begin{equation*}
\cot \theta_{min} := \sup_D \left\{\frac{(c_1(L^{\vee}) - D)^2 - [\omega_0]^2}{2 (c_1(L^{\vee}) -D)\cdot [\omega_0]},\,(c_1(L^{\vee}) -D)\cdot [\omega_0]>0\right\}, 
\end{equation*}
where the supremum is taken over effective $\R$-divisors. 

Note that, if a line bundle $\hat{L}$ satisfies $c_1(\hat{L}^{\vee}) \cdot [\omega_0] > 0$ and supports a smooth solution of the dHYM equation \eqref{dHYMIntro}, then the phase angle $\varphi(X, \hat{L}, [\omega_0]) \in (0, \pi)$ of the mirror special SYZ Lagrangian section is given by
\begin{equation*}
\cot \varphi(X, \hat{L}, [\omega_0]) =  \frac{(c_1( \hat{L}^{\vee}))^2 - [\omega_0]^2}{2 c_1(\hat{L}^{\vee}) \cdot [\omega_0]}. 
\end{equation*}
\begin{thm}\label{UnstableThm} Suppose we are in the situation of Theorem \ref{MainThm} and Definition \ref{BstabDef}, with $X$ is a toric weak del Pezzo surface. Let the Lagrangian $\tilde{\cL}$ be Bridgeland unstable (i.e. not semistable) with respect to $\sigma_{k[\omega_0]}$. Suppose that the corresponding minimal slope divisor $D$ is in fact a $\Q$-divisor. Then, assuming the Arcara-Miles conjecture 
\ref{AMConj} (known for $X = \Bl_p \PP^2$ or $X = \Bl_{p, q} \PP^2$), possibly after increasing $k > 0$, there is a (shifted) SYZ Lagrangian section $\tilde{\cL}_D$, with a morphism 
\begin{equation*}
\tilde{\cL} \to \tilde{\cL}_D
\end{equation*}     
in $\FS(\cY_{q_k}, W(k \omega_0))$, such that $\tilde{\cL}_D$ is Hamiltonian isotopic to a (shifted) SYZ Lagrangian section in the sense of currents, which is a weak solution of the special Lagrangian equation with phase angle $\theta_{min}$.  
\end{thm}
The technical details of the statement and its proof are given in Section \ref{UnstableProofSec} (see also Example \ref{SupportExample}, following \cite{DatarSong_slopes}, Theorem 1.13).
 
\subsection{Some generalisations of Theorem \ref{MainThm}}
It in natural to ask for extensions of Theorem \ref{MainThm} beyond the case of Lagrangian sections, or for a larger class of toric manifolds (not necessarily weak Fano), or beyond the supercritical phase assumption. In general these seem very difficult problems.

Here we discuss obstructions similar to Theorem \ref{MainThm}, conjecturally related to Lagrangian multi-sections, in a very special case, see Corollary \ref{HighRankFSCor} and Proposition \ref{HighRankBInstProp}. 

We give a weak analogue of Theorem \ref{MainThm} for general projective toric manifolds (Theorem \ref{GeneralThm}). 

Additionally, Proposition \ref{LowerPhaseProp} shows an analogue of Theorem \ref{MainThm} in a special, higher dimensional example where the supercritical assumption can be removed.\\

\noindent\textbf{Acknowledgements.} I am grateful to Vamsi Pingali, Carlo Scarpa, Nicol\`o Sibilla, Peng Zhou and especially to Tristan Collins and Richard Thomas for important comments and suggestions.
\section{Background and main result}\label{BackSec}
\subsection{dHYM connections and Nakai-Moishezon criteria} 
Let $X$ be a compact, $n$-dimensional K\"ahler manifold. Following the conventions of \cite{Takahashi_dHYM}, Section 1, given a fixed K\"ahler form $\omega \in [\omega_0]$ and a class $[\alpha_0] \in H^{1,1}(X, \R)$, the \emph{deformed Hermitian Yang-Mills (dHYM) equation} seeks a smooth representative $\alpha \in [\alpha_0]$ such that 
\begin{equation}\label{dHYMIntro}
\Imm\left(e^{-\ii \phi} (\omega + \ii \alpha)^n\right) = 0,
\end{equation} 
where $\phi \in \R$ is a fixed constant. Note that the phase $e^{-\ii \phi}$ is uniquely determined cohomologically, by the condition
\begin{equation*}
\int_X (\omega_0 + \ii \alpha_0)^n \in e^{\ii \phi}\R_{>0}
\end{equation*}
(we always assume that the integral does not vanish).

It is useful to consider explicitly the variant 
\begin{equation}\label{NegativedHYMIntro}
\Imm\left(e^{\ii \phi} (\omega - \ii \alpha)^n\right) = 0,
\end{equation} 
with
\begin{equation*} 
\int_X (\omega_0 - \ii \alpha_0)^n \in e^{-\ii \phi}\R_{>0}.
\end{equation*}
Naturally, the two equations \eqref{dHYMIntro}, \eqref{NegativedHYMIntro} are equivalent under complex conjugation.

\begin{rmk}\label{LYZRmk} Leung, Yau and Zaslow \cite{LeungYauZaslow} showed that, for a local, semi-flat K\"ahler Calabi-Yau torus fibration $(M, \omega_M, \Omega_M)$, with mirror $(W, \omega_W, \Omega_W)$, the dHYM equation \eqref{NegativedHYMIntro} (i.e. with negative sign for $\alpha$) on a holomorphic line bundle $L \to M$, with $[\alpha_0] = c_1(L)$, is precisely equivalent to the special Lagrangian condition for the mirror Lagrangian section $\cL \subset W$, that is,
\begin{equation*}
\omega_W|_{\cL} = 0,\,\Imm\left(e^{-\ii \phi} {\Omega_W}|_{\cL}\right) = 0.
\end{equation*}
In other words, the special Lagrangian condition for $\cL$ is equivalent to the dHYM equation \eqref{dHYMIntro} (i.e. with positive sign for $\alpha$) for the \emph{dual} line bundle $L^{\vee}$, that is, for $[\alpha_0] = c_1(L^{\vee})$.

This correspondence can be generalised to toric compact K\"ahler manifolds, as recalled e.g. in \cite{Chan_survey} and in \cite{CollinsYau_dHYM_momentmap}, Section 9.
\end{rmk}

A systematic study of the dHYM equation on compact K\"ahler manifolds was initiated by Collins, Jacob and Yau (see \cite{CollinsJacobYau, CollinsYau_dHYM_momentmap, JacobYau_special_Lag}), leading in particular to a conjectural equivalence between solvability of the dHYM equation and a certain numerical positivity condition, at least under an assumption known as \emph{supercritical phase}, which we explain in Section \ref{SupercritSec} (see the condition \eqref{SupercritCondition}). This conjecture has now been settled, at least in the projective case, thanks to the work of several authors 
\cite{Takahashi_dHYM, DatarPingali_dHYM, Song_NakaiMoishezon}, after an initial breakthrough due to Chen \cite{GaoChen_Jeq_dHYM} (see also \cite{CollinsSzekelyhidi} for previous results in the toric case). Let us set
\begin{equation*}
\varphi := \frac{n}{2}\pi - \phi.
\end{equation*}
When assuming supercritical phase, we will always suppose that we have
\begin{equation*}
\varphi \in (0, \pi).
\end{equation*}
\begin{thm}[Chu-Li-Takahashi \cite{Takahashi_dHYM}, Corollary 1.5]\label{dHYMThm} Suppose that $X$ is projective. Then, there exists a \emph{supercritical} solution of the dHYM equation \eqref{dHYMIntro} for any, or some, representative $\omega \in [\omega_0]$ if, and only if, for all proper irreducible subvarieties $V \subset X$ we have
\begin{equation}\label{dHYMPosIntro}
\int_V \Rea(\ii \omega_0 + \alpha_0)^{\dim V} - \cot(\varphi)\Imm(\ii \omega_0 +  \alpha_0)^{\dim V} > 0.
\end{equation}
The solution is unique.
\end{thm}
We refer to the condition \eqref{dHYMPosIntro} as \emph{dHYM positivity} or (following \cite{Song_NakaiMoishezon}) as the \emph{dHYM Nakai-Moishezon criterion}.

Note in particular that the existence of supercritical solutions of \eqref{dHYMIntro} does not depend on the choice of K\"ahler form representing $[\omega_0]$. 

\begin{rmk}
We will always assume that $X$ is toric. In this case, it is enough to test \eqref{dHYMPosIntro} on toric submanifolds $V \subset X$. Furthermore, we will always work with a torus-invariant representative $\omega_0$. Then a solution $\alpha \in [\alpha_0]$ is necessarily torus-invariant, by uniqueness.
\end{rmk}
\begin{rmk}\label{SurfaceRmk} When $X$ is a surface, the supercritical assumption is not needed, and it was shown already in \cite{JacobYau_special_Lag} that that the existence of a solution to \eqref{dHYMIntro} on a compact K\"ahler surface (not necessarily projective) is equivalent to the inequality \eqref{dHYMPosIntro} for all irreducible curves $V \subset X$ (this criterion is called ``twisted positivity" in  \cite{Collins_stability}). This is because the dHYM equation \eqref{dHYMIntro} on surfaces is equivalent to a certain complex Monge-Amp\`ere equation, and \eqref{dHYMPosIntro} in this case becomes the genuine Nakai-Moishezon criterion for a corresponding K\"ahler class.
\end{rmk}
The results of Leung-Yau-Zaslow mentioned in Remark \ref{LYZRmk} are closely related to the Thomas-Yau conjecture on stability and the existence of special Lagrangians. Thus, in the toric case, one can try to use Theorem \ref{dHYMThm}, together with known mirror symmetry results, in order to obtain applications to the (non)existence of special Lagrangians. 

This problem was first studied by Collins and Yau in \cite{CollinsYau_dHYM_momentmap}, Section 9 (before Theorem \ref{dHYMThm} was proved), where an approach based on certain one-parameter degenerations of Lagrangian submanifolds was proposed. Note that, as we will make clear, Collins and Yau work with a class of Lagrangians which is very similar to the one we will consider, i.e. Strominger-Yau-Zaslow (SYZ) transforms of a high power of an ample line bundle $L^{\otimes k}$, close to the large volume limit $k[\omega_0]$, for $k \gg 1$ (see Remark \ref{CollinsYauRmk}).

As the authors explain, their approach, while natural and geometric, involves certain key technical difficulties. 

\subsection{Toric weak Fano manifolds} 
As in the Introduction, let us assume that $X$ is a weak Fano manifold, with mirror Landau-Ginzburg model $(\cY \to \cM,\,W\!: \cY \to \C)$ (see Section \ref{MirrorSec}).
\begin{thm}\label{FanoThm} Let $X$ be a toric weak Fano manifold with a fixed K\"ahler class $\omega_0$ and a holomorphic line bundle $L$. For all sufficiently large, fixed $k > 0$, there exist classes
\begin{equation*}
\Gamma,\,\Gamma_V \in H_n(\cY_{q_k}, \{\Rea(W(k\omega_0)) \gg 0\}; \Z),
\end{equation*} 
where $V$ ranges over all toric submanifolds of $X$, such that, if $[\omega_0]$ is \emph{generic} in the sense that it does not lie in the union of finitely many proper analytic subvarieties of $H^{1, 1}(X, \R)$ determined by $c_1(L)$, then the dHYM positivity condition \eqref{dHYMPosIntro} for $(X,  [\omega_0], c_1(L^{\vee}))$ is equivalent to the phase inequalitites for periods 
\begin{equation*}
\arg\left( \int_{\Gamma_V} (-1)^{\codim V}e^{-W(k\omega_0)}\Omega_0 \right)< \arg \int_{\Gamma} e^{-W(k\omega_0)}\Omega_0. 
\end{equation*}
Moreover, we have
\begin{equation*}
\Gamma = \Gamma(L^{\otimes k}),\,\Gamma_V = \Gamma(\cS_V)
\end{equation*}
for certain objects $\cS_V \in D^b(X)$, with morphisms
\begin{equation*}
\cS_V[-\codim V] \to L^{\otimes k},  
\end{equation*}
where   
\begin{equation*}
\Gamma\!: K^0(X) \xrightarrow{\sim} H_n(\cY_{q_k}, \{\Rea(W(k\omega_0)) \gg 0\}; \Z)
\end{equation*}
denotes Iritani's isomorphism (see Theorem \ref{GammaThm}). Finally, $k > 0$ can be chosen uniformly for $[L]$ lying in a bounded subset of $K^0(X)$.
\end{thm}
The proof is obtained by a suitable scaling analysis of the toric $\widehat{\Gamma}$-theorem of Iritani \cite{Iritani_gamma}. The $\widehat{\Gamma}$-theorem is discussed in Section \ref{GammaSec} (and stated as Theorem \ref{GammaThm}). Theorem \ref{FanoThm} is proved in Section \ref{FanoThmSec}. 
\begin{rmk}\label{FanoThmRmk} Some comments on Theorem \ref{FanoThm} are in order.
\begin{enumerate}    
\item[$(i)$] Ultimately our scaling analysis is made possible by the following simple scale-invariance property of the (supercritical) dHYM equation: if $(X, \omega, \alpha)$ is a (supercritical) solution of \eqref{dHYMIntro}, then for all $k > 0$ the same holds for $(X, k\omega, k\alpha)$ (see Section \ref{SupercritSec}). 

\item[$(ii)$] It will be clear from the proof of Theorem \ref{FanoThm} that, with the same assumptions and notation, the dHYM positivity condition \eqref{dHYMPosIntro} for $(X,  [\omega_0], c_1(L^{\vee}))$ can be written as 
\begin{equation*}
\arg\left( (-1)^{\codim V}\int_{X} e^{-\ii k\omega_0}\ch(\cS_V)\right)< \arg \int_{X} e^{-\ii k\omega_0}\ch(L^{\otimes k}), 
\end{equation*}
see Remark \ref{CentralChargeRmk}. We will apply this in Section \ref{BSec} in order to describe a relation with Bridgeland stability conditions when $X$ is a toric weak del Pezzo surface, and in particular to prove Theorem \ref{BstabThm}.
\item[$(iii)$] The proof of Theorem \ref{FanoThm} shows that, in principle, the scale factor $k > 0$ could be controlled quantitatively in terms of Givental's $I$-function $I(q, -z)$ for $X$ and so in terms of genus $0$ Gromov-Witten invariants, virtually enumerating rational curves contained in $X$ (we recall the definition of $I(q, -z)$ in Section \ref{GammaSec}). The role of holomorphic curves in the Thomas-Yau conjecture has been emphasised by Li \cite{YangLi_ThomasYau}.    
\item[$(iv)$] The fact that $k > 0$ can be chosen uniformly as $V\subset X$ varies is essential, since it means that we are working on a \emph{fixed} LG model. In the toric case this uniform choice is possible since there are only finitely many irreducible toric subvarieties and so, a priori, only finitely many possible ``destabilisers", i.e. $V \subset X$ violating \eqref{dHYMPosIntro}. Khalid and Sj\"ostr\"om Dyrefelt prove several results on the general finiteness problem for destabilizers, see e.g. \cite{SohaibZak_higherdim}.
\item[$(v)$] The proof of Theorem \ref{FanoThm} shows that the genericity condition required on $[\omega_0]$ is precisely the absence of strict semi-stabilisers $V$ for the Nakai-Moishezon criterion \eqref{dHYMPosIntro}, i.e. $V \subset X$ for which the left hand side of \eqref{dHYMPosIntro} vanishes.   
\end{enumerate}
\end{rmk}
As explained in the Introduction, we will apply homological mirror symmetry for toric manifolds in the sense of Fang-Liu-Treumann-Zaslow \cite{Zaslow_toricHMS} and its non-equivariant version (see e.g. \cite{KuwagakiCohCon, SibillaCohCon, ZhouCohCon}). Its compatibility with the $\widehat{\Gamma}$-theorem is proved in the work of Fang \cite{Fang_charges}. Through these results, Theorem \ref{FanoThm} has the following immediate consequence.
\begin{cor}\label{FSCor} In the situation of Theorem \ref{FanoThm}, there exists a \emph{supercritical} solution to the dHYM equation \eqref{dHYMIntro} on $L^{\vee}$ iff 
\begin{equation*}
\arg\left((-1)^{\codim V}\int_{[\cL_V]} e^{-W(k\omega_0)}\Omega_0 \right)< \arg \int_{[\cL]} e^{-W(k\omega_0)}\Omega_0, 
\end{equation*}
where $\cL$, $\cL_V$ correspond to $L^{\otimes k},\,\cS_V \in D^b(X)$ under the mirror isomorphism 
\begin{equation*}
D^b(X) \cong \FS(\cY_{q_k}, W(k \omega_0)),
\end{equation*}
so that there are morphisms
\begin{equation*}
\cL_V[-\codim V] \to \cL, 
\end{equation*}
and $[\cL]$, $[\cL_V]$ denote their classes in $H_n(\cY_{q_k}, \{\Rea(W(k\omega_0))\gg 0\}; \Z)$ (in the sense of Fang \cite{Fang_charges}).

When $X$ is a toric weak del Pezzo surface, the same result holds without the supercritical assumpion.
\end{cor}
\begin{rmk}\label{CollinsYauRmk} The geometric approach to the same problem proposed by Collins and Yau \cite{CollinsYau_dHYM_momentmap} Section 9, using one-parameter degenerations of a fixed Lagrangian $\tilde{\cL}$, also involves scaling the K\"ahler form and the Lagrangian by a possibly very large positive parameter. From their viewpoint, this occurs because the limiting Lagrangian $\tilde{\cL}_{\infty}$ only defines an object of the Fukaya-Seidel category after this scaling. In comparison, as will be clear from the proof of Theorem \ref{FanoThm}, the scaling we use is due to our application of the $\widehat{\Gamma}$-theorem.   
\end{rmk}
It follows from the work of Fang \cite{Fang_charges} (see in particular his Remark 1.3) that, when $L^{\vee}$ is ample, the cycles
\begin{equation*}
[\cL^{\vee}] = \Gamma(L^{- k}),\,[\cL] = \Gamma(L^{\otimes k})
\end{equation*}
can be represented by the Strominger-Yau-Zaslow (SYZ) transforms $(L^{- k}, h^{- k})$, respectively $(L^{\otimes k}, h^{\otimes k})$, where $h$ denotes any Hermitian metric on the fibres of $L$. We denote the latter SYZ transform by
\begin{equation*}
\cL_{\syz}  = \cL_{\syz}(L^{\otimes k}, h^{\otimes k}).
\end{equation*}
The local results of Leung-Yau-Zaslow \cite{LeungYauZaslow} (see Remark \ref{LYZRmk}), adapted to the compact toric case (see e.g. \cite{Chan_survey} and \cite{CollinsYau_dHYM_momentmap}, Section 9), provide a holomorphic volume form $\Omega(k \omega_0)$, mirror to $k\omega_0$, such that $\cL_{\syz}$ is special Lagrangian with respect to $\Omega(k \omega_0)$ iff $(L^{-  k}, h^{-  k})$ satisfies the dHYM equation \eqref{dHYMIntro}, i.e. on the \emph{dual} line bundle, with respect to the background metric $k\omega_0$. Thus, it is natural to call \emph{supercritical} a special Lagrangian section corresponding to a supercritical solution of \eqref{dHYMIntro} on the dual line bundle.

The supercritical condition and solvability of \eqref{dHYMIntro} are invariant under a common rescaling $([\omega_0], [\alpha_0]) \mapsto (k[\omega_0], k[\alpha_0])$ for $k > 0$ (see Section \ref{SupercritSec}). 

Therefore, we obtain the following application of Theorem \ref{FanoThm}.
\begin{cor}\label{ThomasYauCor} In the situation of Corollary \ref{FSCor}, if $L^{\vee}$ is ample, then 
\begin{equation*}
\cL \in \FS(\cY_{q_k}, W(k \omega_0))
\end{equation*}
satisfies a form of the Thomas-Yau conjecture: namely, there exists a \emph{supercritical special Lagrangian Strominger-Yau-Zaslow representative} $\cL_{\syz}$ of $ \cL $  with respect to $\Omega(k \omega_0)$ iff the morphisms  
\begin{equation*}
\cL_V[-\codim V] \to \cL  
\end{equation*} 
in $\FS(\cY_{q_k}, W(k \omega_0))$ do not destabilise $\cL$, in the sense that the phase inequalities for periods
\begin{equation*}
\arg\left((-1)^{\codim V}\int_{[\cL_V]} e^{-W(k\omega_0)}\Omega_0 \right) < \arg \int_{\cL} e^{-W(k\omega_0)}\Omega_0 
\end{equation*}
hold.

When $n = 2$, i.e. if $X$ is a toric weak del Pezzo surface, the supercritical assumption is not necessary.    
\end{cor}
\begin{rmk}\label{VolumeFormRmk} We can state Corollary \ref{ThomasYauCor} in a way which is perhaps more natural by using the SYZ transform $\mathcal{F}^{\syz}$ on differential forms (see \cite{Chan_survey, ChanLeung_SYZ}). Since this is surjective, there exist representatives $\omega_k \in [\omega_0]$ such that, for all $k > 0$, we have
\begin{equation*}
\mathcal{F}^{\syz}(e^{\ii k \omega_k}) = \Omega_0. 
\end{equation*}
On the other hand, according to \cite{ChanLeung_SYZ}, Theorem 1.1 (see also \cite{Chan_survey}, Theorem 3.6), we also have
\begin{equation*}
\mathcal{F}^{\syz}(e^{\ii k \omega_k}) = e^{W(k \omega_k)} \Omega(k \omega_k)(1 + O(k^{-1})) = e^{W(k \omega_0)} \Omega(k \omega_k)(1 + O(k^{-1})),  
\end{equation*}
using the fact the LG potential $W(k \omega_0 )$ only depends on the cohomology class. So, we find
\begin{equation*}
\Omega(k \omega_k) = e^{-W(k \omega_0)} \Omega_0 (1 + O(k^{-1})). 
\end{equation*}
At the same time, the SYZ Lagrangian $\cL_{\syz}$ obtained from $(L^{\otimes k}, h^{\otimes k})$ using the K\"ahler form $\omega_k$ is special Lagrangian with respect to $\Omega(k \omega_k) $ iff the supercritical dHYM equation is solvable for $[\alpha_0] = - k c_1(L)$ with respect to the background metric $k\omega_k$. However, as we observed, this is equivalent to solvability for $[\alpha_0] = c_1(L)$ with respect to $\omega_0$. 

Thus, we can restate Corollary \ref{ThomasYauCor} as follows: \emph{for all sufficiently large $k > 0$, there exists a supercritical special Lagrangian Strominger-Yau-Zaslow representative $\cL_{\syz}$ of $\cL$  with respect to $\Omega(k \omega_k)$ iff the morphisms  
\begin{equation*}
\cL_V[-\codim V] \to \cL  
\end{equation*} 
in $\FS(\cY_{q_k}, W(k \omega_k))$ do not destabilise $\cL$, in the sense that the phase inequalities for periods
\begin{equation*}
\arg\left((-1)^{\codim V}\int_{[\cL_V]} \Omega(k \omega_k) \right) < \arg \int_{ \cL } \Omega(k \omega_k)
\end{equation*}
hold. As we noticed, the supercritical assumption can be removed for toric weak del Pezzo surfaces.}  
\end{rmk}
\begin{rmk} Li's inequality \eqref{LiPhaseInequ} is proved for \emph{compact} special Lagrangians (under further assumptions). The Strominger-Yau-Zaslow representative $\cL_{\syz}$ is noncompact. So it is not immediately clear how to compare \eqref{LiPhaseInequ} with the phase inequalities in Corollary \ref{ThomasYauCor}. 
\end{rmk}
\begin{rmk} Versions of the $\widehat{\Gamma}$-theorem (Theorem \ref{GammaThm}), i.e. the ``mirror-symmetric Gamma conjecture", are also known for some non-toric or non-Fano manifolds. The work \cite{Iritani_SYZGamma} emphasises the connection of these results with the Strominger-Yau-Zaslow approach to mirror symmetry. 
\end{rmk}
\subsection{Lower phase}
The supercritical phase assumption can also be removed in a few higher-dimensional cases. Here we only discuss the special example $X = \Bl_p \PP^n$.  
\begin{prop}\label{LowerPhaseProp} In the situation of Corollary \ref{FSCor}, suppose $X = \Bl_p \PP^n$. Assume $L^{\vee}$ is ample. Then there exists a unique constant $\hat{\theta} \in \R$, satisfying
\begin{equation*}
\hat{\theta} = -\arg\left(-(-2\pi \ii)^{-n} \int_{\cL} e^{-W(k \omega_0)/z} \Omega_0\right) + O(k^{-1}) \mod 2\pi,
\end{equation*}
such that $\cL$ admits a special Lagrangian Strominger-Yau-Zaslow representative $\cL_{\syz}$ with respect to $\Omega(k \omega_0)$ iff the phase inequalities for periods
\begin{equation*}
\hat{\theta} - \frac{\pi}{2} < -\arg\left((-2\pi \ii)^{-n} \int_{[\cL_V]} e^{-W(k \omega_0)/z} \Omega_0\right) < \hat{\theta} + \frac{\pi}{2} 
\end{equation*}
hold for all toric \emph{divisors} $V$.
\end{prop}
Proposition \ref{LowerPhaseProp} is an application of a result of Jacob and Sheu \cite{JacobSheu}. The proof is given in Section \ref{LowerPhaseSec}.
\subsection{Higher rank}
An analogue of the dHYM equation \eqref{dHYMIntro} for a higher rank holomorphic vector bundle $E\to X$ is discussed in \cite{CollinsYau_momentmaps_preprint}, Section 8.1, following proposals in the physics literature (see e.g. \cite{DirichletBook}, Section 5.2.2.2). This is given explicitly by 
\begin{equation}\label{HigherRankdHYM}
\Imm\left( e^{-\ii \hat{\phi}_{[\omega_0]}(E)} \big(\omega \otimes \operatorname{Id}_E  - F_E(h))^n\right) = 0,
\end{equation}
where $F_E(h)$ denotes the curvature of the Chern connection of a hermitian metric $h$ on the fibres of $E$, the notation $\Imm$ denotes the skew-Hermitian part of an endomorphism with respect to the metric $h$, and $e^{-\ii \hat{\phi}_{[\omega_0]}(E)}$ is uniquely determined by integrating the trace (and only depends on the topology of $E$).

Solutions $(E, h)$ of \eqref{HigherRankdHYM} should be related to special Lagrangian multi-sections, although we do not know a precise result.

\begin{rmk}Examples from \cite{Dervan_Zconnections, KellerScarpa} suggest that in the higher rank case it is especially important to allow a $B$-field, namely, a fixed representative $\beta$ of a $(1, 1)$-class $[\beta_0]$, deforming \eqref{HigherRankdHYM} to
\begin{equation}\label{BfieldHigherRankdHYM}
\Imm\left( e^{-\ii \hat{\phi}_{[\omega_0], [\beta_0]}(E)} \big((\omega - \ii \beta)\otimes \operatorname{Id}_E  - F_E(h))^n\right) = 0.
\end{equation}
\end{rmk}

The higher rank dHYM equation \eqref{HigherRankdHYM} and a class of closely related PDEs (such as vector-bundle complex Monge-Amp\`ere and $Z$-critical equations) have been studied is several works, including \cite{Dervan_Zconnections, Pingali_vectorMA}. Here we focus on the obstructions proved in \cite{KellerScarpa} in a special case (see also \cite{Takahashi_Jequ} for related results). 

Let $E \to X$ denote a rank $2$ holomorphic vector bundle over a compact K\"ahler surface $X$. Let us denote by $V \subset X$ any irreducible curve, and by $L \subset E$ a proper line sub-bundle. Fix a K\"ahler class $[\omega_0]$ and $(1, 1)$-classes $[\alpha_0]$, $[\beta_0]$. Let us set 
\begin{align*}
& Z^{\dhym}_X(E) := -\ii \int_X e^{-\ii \omega_0} e^{-\beta_0} \ch(E),\,Z^{\dhym}_X(L) := -\ii \int_X e^{-\ii \omega_0} e^{-\beta_0} \ch(L),\\
& Z^{\dhym}_V(E|_{V}) := -\ii \int_V e^{-\ii \omega_0} e^{-\beta_0} \ch(E|_{V}). 
\end{align*}

Keller and Scarpa \cite{KellerScarpa} introduce a space $\cH^+(E)$ of \emph{$Z^{\dhym}$-positive hermitian metrics}, a notion we recall here in Section \ref{HigherRankPosSec}; this is the same as the space of \emph{subsolutions} of the dHYM equation studied in \cite{Dervan_Zconnections}. Roughly speaking, the dHYM equation \eqref{HigherRankdHYM} is only elliptic on $\cH^+(E)$.   
\begin{thm}[Keller-Scarpa \cite{KellerScarpa}]\label{KellerScarpaThm} Let $E \to X$ be as above. Suppose there exists a $Z^{\dhym}$-positive hermitian metric $h \in \cH^+(E)$ (see Definition \ref{HighRankPositivityDef}). Then, for all $V$, we have 
\begin{equation}\label{HighRankPositivity}
\arg\left(Z^{\dhym}_X(E)\right) < \arg\left(Z^{\dhym}_V(E|_V)\right) < \arg\left(Z^{\dhym}_X(E)\right) + \pi.
\end{equation} 
Moreover, if $E$ is indecomposable and $h \in \cH^+(E)$ solves the higher rank dHYM equation with $B$-field \eqref{BfieldHigherRankdHYM}, then, for all line sub-bundles $L \subset E$, we have
\begin{equation}\label{HighRankStability}
\arg\left(Z^{\dhym}_X(E)\right) - \pi < \arg\left(Z^{\dhym}_X(L) \right) < \arg\left(Z^{\dhym}_X(E)\right).
\end{equation}  
\end{thm} 
Note that in fact \cite{KellerScarpa} contains a more general result, applying to a larger class of PDEs (containing the dHYM equation with $B$-field), and yielding stronger obstructions.

In the situation of Theorem \eqref{KellerScarpaThm}, in certain cases, we can use the $\widehat{\Gamma}$-theorem to translate the positivity and stability conditions \eqref{HighRankPositivity}, \eqref{HighRankStability} into phase inequalities for periods of Lagrangians. We denote by $W( \omega - \ii \beta )$ the LG potential of a complexified K\"ahler class. 
\begin{thm}\label{HighRankFanoThm} Let $X$ be a toric weak del Pezzo surface. Suppose that, for all $k > 0$, there are bundles $E_k \to X$ given by nontrivial extensions 
\begin{equation*}
0 \to L^{\otimes k}_1 \to E_k \to L^{\otimes k}_2 \to 0,
\end{equation*}
for fixed line bundles $L_i \to X$. Let $E := E_1$. 

If there exists $h \in \cH^+(E)$ solving the higher rank dHYM equation with $B$-field \eqref{BfieldHigherRankdHYM} on $E$, then, for all sufficiently large $k > 0$, if $[\omega_0]$ is \emph{generic} in the sense that it does not lie in the union of finitely many proper analytic subvarieties of $H^{1, 1}(X, \R)$ determined by $c_1(L_i)$, $i =1, 2$, we have
\begin{align*}
&\arg \int_{\Gamma} e^{-W(k(\omega_0 - \ii \beta_0)) }\Omega_0 <  \arg\left(\int_{\Gamma_V} e^{-W(k(\omega_0 - \ii \beta_0)) }\Omega_0 \right) < \arg \int_{\Gamma} e^{-W(k(\omega_0 - \ii \beta_0)) }\Omega_0 + \pi,\\
&\arg \int_{\Gamma} e^{-W(k(\omega_0 - \ii \beta_0)) }\Omega_0 - \pi <  \arg\left( \int_{\Gamma_{L_1}} e^{-W(k(\omega_0 - \ii \beta_0)) }\Omega_0 \right) < \arg \int_{\Gamma} e^{-W(k(\omega_0 - \ii \beta_0)) }\Omega_0,
\end{align*}
where we let
\begin{equation*}
\Gamma = \Gamma(E_k),\,\Gamma_V = \Gamma(\cS_V),\,\Gamma_{L_1} = \Gamma(L^{\otimes k}_1)
\end{equation*}
for $L^{\otimes k}_1 \in D^b(X)$ and for certain objects $\cS_V \in D^b(X)$, with morphisms
\begin{equation*}
\cS_V[-1] \to E_k,\, L^{\otimes k}_1 \to E_k,  
\end{equation*}
and   
\begin{equation*}
\Gamma\!: K^0(X) \xrightarrow{\sim} H_n(\cY_{q_k}, \{\Rea(W(k(\omega_0 - \ii \beta_0)))\gg 0\}; \Z)
\end{equation*}
denotes Iritani's isomorphism (see Theorem \ref{GammaThm}). 
\end{thm}
Theorem \ref{HighRankFanoThm} is proved in Section \ref{HighRankProofSec}. Through toric homological mirror symmetry \cite{Abouzaid_toricHMS, Zaslow_toricHMS}, this yields the following analogue of Corollary \ref{FSCor}. 
\begin{cor}\label{HighRankFSCor} In the situation of Theorem \ref{HighRankFanoThm}, if the higher rank dHYM equation \eqref{BfieldHigherRankdHYM} has a positive solution $h \in \cH^+(E)$ on $E$, then, for all sufficiently large $k > 0$, we have
\begin{align*} 
&\nonumber \arg \int_{[\cL]} e^{-W(k(\omega_0 - \ii \beta_0))}\Omega_0 < \arg\left( \int_{[\cL_V]} e^{-W(k(\omega_0 - \ii \beta_0)) }\Omega_0 \right)  < \arg \int_{[\cL]} e^{-W(k(\omega_0 - \ii \beta_0))}\Omega_0 + \pi,\\
& \arg \int_{[\cL]} e^{-W(k(\omega_0 - \ii \beta_0))}\Omega_0 - \pi < \arg\left( \int_{[\cL_1]} e^{-W(k(\omega_0 - \ii \beta_0))}\Omega_0 \right) < \arg \int_{[\cL]} e^{-W(k(\omega_0 - \ii \beta_0))}\Omega_0,
\end{align*}
where $\cL$ is the image of $E_k \in D^b(X)$ under the mirror isomorphism 
\begin{equation*}
D^b(X) \cong \FS(\cY_{q_k}, W(k ((\omega_0 - \ii \beta_0)))),
\end{equation*}
$\cL_V,\,\cL_1 \in \FS(\cY_{q_k}, W(k (\omega_0 - \ii \beta_0)))$ are mirror to $\cS_V$, $L^{\otimes k}_1$, with morphisms
\begin{equation*}
\cL_V[-1] \to \cL,\,\cL_1 \to \cL, 
\end{equation*}
and $[\cL]$, $[\cL_V]$, $[\cL_1]$ denote their classes in $H_n(\cY_{q_k}, \{\Rea(W(k(\omega_0 - \ii \beta_0)))\gg 0\}; \Z)$.
\end{cor}
A concrete example of bundles $E_k$ on $X = \Bl_p\PP^2$ to which this result applies is given in Section \ref{KellerScarpaExm}, following \cite{KellerScarpa}. 

Under suitable assumptions, $\cL$ should be quasi-isomorphic to a Lagrangian multi-section, and one could hope that its class in $\FS(\cY_{q_k}, W(k (\omega_0 - \ii \beta_0)))$ can be represented by a smooth special Lagrangian multi-section iff the morphisms $\cL_V[-1] \to \cL$, $\cL_1 \to \cL$ do not destabilise $\cL$, in the sense that the phase inequalities for periods appearing in Corollary \ref{HighRankFSCor} hold. 

In this setup, there is a natural analogue of Definition \ref{BstabDef}.
\begin{definition}\label{HighRankBstabDef} In the situation of Theorem \ref{HighRankFanoThm} and Corollary \ref{HighRankFSCor}, suppose that the $k[\beta_0]$-twisted slopes of $L^{\otimes k}_i$ are nonpositive, $\mu_{k[\omega_0], k[\beta_0]}(L^{\otimes k}_i) \leq 0$, $i=1,2$.

Then, we say that the Lagrangian $\tilde{\cL} := \cL[1]$ is Bridgeland unstable, iff its mirror $E_k[1]$ is unstable with respect to the stability condition
\begin{equation*}
\sigma_{k [\omega_0], k[ \beta_0]} = \left(\Coh^{\sharp}_{k[\beta_0]}(X), Z_{k [\omega_0], k[ \beta_0]}\right).
\end{equation*} 
\end{definition}
The definition of $\sigma_{k [\omega_0], k[ \beta_0]}$ is recalled in Section \ref{HighRankBInstPropProof}. We will see that Definition \ref{HighRankBstabDef} is consistent (Section \ref{HighRankBInstPropProof}) and has the correct central charge (see \eqref{AsympEk}). We also have an analogue of Theorem \ref{BstabThm}. 
\begin{prop}\label{HighRankBInstProp} In the situation of Theorem \ref{HighRankFanoThm}, fix $k >0$ sufficiently large, and suppose that the $k[\beta_0]$-twisted slopes of $L^{\otimes k}_i$ are nonpositive, $\mu_{k[\omega_0], k[\beta_0]}(L^{\otimes k}_i) \leq 0$, $i=1,2$. 

If $\cL_V$ or $\cL_1$ violate the phase inequalities for periods appearing in Corollary \ref{HighRankFSCor}, then $\tilde{\cL}$ is Bridgeland unstable, destabilised by $\cL_V$ or $\cL_1$.
\end{prop}
Proposition \ref{HighRankBInstProp} is proved in Section \ref{HighRankBInstPropProof}.
\subsection{General toric manifolds}
Let $X$ be a projective toric manifold of dimension $n$, with a fixed (complexified) K\"ahler class $[\omega_0]$. Hodge-theoretic mirror symmetry in this generality is studied in \cite{CoatesCortiIritani_hodge}, and we recall some of its key properties in Section \ref{MirrorSec}. It yields a mirror family $\cY \to \cM$ with a Landau-Ginzburg potential $W\!: \cY \to \C$. There is a fibre $\cY_q \cong (\C^*)^n$ of the mirror family corresponding to $[\omega_0]$. We prove that this satisfies a weak analogue of Theorem \ref{FanoThm}.
\begin{thm}\label{GeneralThm} Let $X$ be a projective toric manifold with a fixed K\"ahler class $[\omega_0]$ and a line bundle $L \to X$. 

For all sufficiently large $k > 0$, if $[\omega_0]$ is \emph{generic} in the sense that it does not lie in the union of finitely many proper analytic subvarieties of $H^{1, 1}(X, \R)$ determined by $c_1(L)$, then there exist \emph{complex} cycles $\Gamma_{L}$, $\Gamma_{L, V}$, as $V$ ranges through toric submanifolds of $X$, representing classes
\begin{equation*}
[\Gamma_{L}],\,[\Gamma_{L, V}] \in H_n(\cY_{q_k}, \{\Rea(W(k\omega_0)) \gg 0\}; \Z) \otimes \C,
\end{equation*}
satisfying
\begin{equation*}
\operatorname{supp}(\Gamma_{L, V}) \subsetneq \operatorname{supp} \Gamma_{L},
\end{equation*}
and a holomorphic volume form $\Omega^{(k)}$, such that the dHYM positivity condition \eqref{dHYMPosIntro} for $(X, [\omega_0], c_1(L^{\vee}))$ is equivalent to 
\begin{equation}\label{GeneralThmPeriodsIneq}
\arg\left((-1)^{\codim V} \int_{\Gamma_{L, V}} e^{-W(k\omega_0)/z}\Omega^{(k)}\right) < \arg \int_{\Gamma_L} e^{-W(k\omega_0)/z}\Omega^{(k)}  
\end{equation}
where the integrals are understood in the sense of their asymptotic expansion as $z \to 0^+$ (as in \cite{CoatesCortiIritani_hodge}, Section 6.2). 

As a consequence, a SYZ Lagrangian section $\cL(L^{\otimes k}, h)$ is Hamiltonian isotopic to a supercritical SYZ Lagrangian section iff the phase inequalities for periods \eqref{GeneralThmPeriodsIneq} hold. 
\end{thm}
The proof of this result, provided in Section \ref{GeneralToricProofSec}, uses properties of the  general mirror maps constructed in \cite{CoatesCortiIritani_hodge}. 

However in this general toric case we do not know if the phase inequalities \eqref{GeneralThmPeriodsIneq} correspond to exact triangles in $\FS(\cY_{q_k}, W(k \omega_0))$. In fact we do not even know how to relate the integration cycle $\Gamma_L$ to the Lagrangian section $\cL$.
\section{dHYM positivity}\label{PositivitySec}    
In this Section we provide a few more technical details concerning deformed Hermitian Yang-Mills connections.
\subsection{Supercritical phase}\label{SupercritSec}
Consider $(X, [\omega_0], [\alpha_0])$ as in the Introduction. We follow the conventions of \cite{Takahashi_dHYM}, Section 1. For fixed representatives $\omega \in [\omega_0]$, $\alpha \in [\omega_0]$ we can form the endomorphism
\begin{equation*}
\omega^{-1} \alpha \in \operatorname{End}(T^{1,0} X),
\end{equation*}
with eigenvalues $\lambda_1, \ldots, \lambda_n$. The \emph{Lagrangian phase operator} is defined as
\begin{equation*}
\vartheta(\omega^{-1} \alpha) := \sum^n_{i = 1} \operatorname{arccot}(\lambda_i).
\end{equation*}
We denote by $\varphi$ a real constant.
\begin{definition} The \emph{supercritical dHYM equation} is the special case of the \emph{constant Lagrangian phase equation} given by
\begin{equation}\label{SupercritCondition}
\vartheta(\omega^{-1} \alpha) = \varphi \in (0, \pi). 
\end{equation} 
\end{definition}
As first proved in \cite{JacobYau_special_Lag}, the constant Lagrangian phase equation actually implies that $\alpha$ solves the dHYM equation \eqref{dHYMIntro}.

Thus, the set of supercritical solutions of \eqref{dHYMIntro} is invariant under a rescaling $(\omega, \alpha) \mapsto (k \omega, k\alpha)$. Clearly, the same holds for the phase $e^{\ii \phi}$, and for the solvability of the dHYM equation \eqref{dHYMIntro}.
\subsection{Phase inequality}\label{PhaseIneqSec} 
Suppose $[\alpha_0] = c_1(L)$. As observed, for example, in \cite{CollinsYau_momentmaps_preprint}, Section 8, we have   
\begin{equation*}
\int_X (\omega_0 + \ii c_1(L))^n = -e^{\ii (n - 2)\frac{\pi}{2}} \int_X e^{-\ii \omega_0}\ch(L),
\end{equation*}
and, more generally, for $V \subset X$, 
\begin{equation*}
\int_V (\omega_0 + \ii c_1(L))^{\dim V} = -e^{\ii (\dim V  - 2)\frac{\pi}{2}} \int_V e^{-\ii \omega_0}\ch(L).
\end{equation*}
Using these identities, one can show that the dHYM positivity condition \eqref{dHYMPosIntro}, in the supercritical phase, can be written equivalently as
\begin{equation*}
\arg\left(-\int_V e^{-\ii \omega_0}\ch(L)\right) > \arg\left(-\int_X e^{-\ii \omega_0}\ch(L)\right) 
\end{equation*}
(see e.g. \cite{YangLi_ThomasYau}, Section 2.6).

Applying this to the dual line bundle $L^{\vee}$, we see that $L^{\vee}$ supports a supercritical solution of \eqref{dHYMIntro} iff 
\begin{equation*} 
\arg\left(-\int_V e^{-\ii \omega_0}\ch(L^{\vee})\right) > \arg\left(-\int_X e^{-\ii \omega_0}\ch(L^{\vee})\right).
\end{equation*}
This condition can be written in terms of $L$ as
\begin{equation*} 
\arg\left(-(-1)^{\dim V}\int_V e^{\ii \omega_0}\ch(L)\right) > \arg\left(-(-1)^n\int_X e^{\ii \omega_0}\ch(L)\right),
\end{equation*}
and so after complex conjugation as
\begin{equation*} 
\arg\left(-(-1)^{\dim V}\int_V e^{-\ii \omega_0}\ch(L)\right) < \arg\left(-(-1)^n\int_X e^{-\ii \omega_0}\ch(L)\right),
\end{equation*}
or equivalently
\begin{equation}\label{PhaseIneq} 
\arg\left((-1)^{\codim V}\int_V e^{-\ii \omega_0}\ch(L)\right) < \arg\left(\int_X e^{-\ii \omega_0}\ch(L)\right).
\end{equation}
\subsection{Higher rank}\label{HigherRankPosSec}
Let $E \to X$ denote a rank $2$ vector bundle over a compact K\"ahler surface $X$. Write $V \subset X$ for any irreducible curve.

Keller and Scarpa \cite{KellerScarpa} formulate their result in terms of the conditions 
\begin{equation*}
\Imm\left( \frac{Z^{\dhym}_V(E|_V)}{Z^{\dhym}_X(E)} \right) > 0,\,\Imm\left( \frac{Z^{\dhym}_X(L)}{Z^{\dhym}_X(E)} \right) < 0, 
\end{equation*} 
which can be written as 
\begin{align*}
&\arg\left(Z^{\dhym}_X(E)\right) < \arg\left(Z^{\dhym}_V(E|_V)\right) < \arg\left(Z^{\dhym}_X(E)\right) + \pi,\\  
&\arg\left(Z^{\dhym}_X(E)\right) - \pi < \arg\left(Z^{\dhym}_X(L) \right) < \arg\left(Z^{\dhym}_X(E)\right). 
\end{align*} 
Explicitly, these are given by
\begin{align}\label{HighRankPhaseInequ}
&\nonumber \arg\left(\int_X e^{-\ii \omega_0} e^{-\beta_0} \ch(E)\right) -\pi\\
&\nonumber< \arg\left(\int_V e^{-\ii \omega_0} e^{-\beta_0} \ch(E|_{V})\right) - \pi < \arg\left(\int_X e^{-\ii \omega_0} e^{-\beta_0} \ch(E)\right),\\  
& \arg\left(\int_X e^{-\ii \omega_0} e^{-\beta_0} \ch(E)\right) - \pi < \arg\left(\int_X e^{-\ii \omega_0} e^{-\beta_0} \ch(L)\right) < \arg\left(\int_X e^{-\ii \omega_0} e^{-\beta_0} \ch(E)\right). 
\end{align} 

Recall that the results of \cite{KellerScarpa} hinge on a suitable positivity (or subsolution) condition.
\begin{definition}[\cite{KellerScarpa}, Definition 2.5]\label{HighRankPositivityDef} The space $\cH^+(E)$ of \emph{$Z^{\dhym}$-positive} hermitian metrics on $E$ (also known as $Z^{\dhym}$-subsolutions) is given by metrics $h$ on $E$ such that   
\begin{equation*}
\Imm\left( e^{-\ii \hat{\phi}_{[\omega_0]}([E, \beta_0])} \frac{\del}{\del F_E(h)}\big((\omega - \ii \beta)\otimes \operatorname{Id}_E  - F_E(h))^n\right)
\end{equation*}
is a positive $2(n-1)$ $\End(E)$-valued form.
\end{definition}
\section{Hodge theoretic mirror symmetry}\label{MirrorSec}
We collect here some key facts about Hodge-theoretic mirror symmetry for toric manifolds.

Suppose $X$ is a projective toric manifold. Following \cite{CoatesCortiIritani_hodge} (which also contains an extensive list of references), the mirror to $X$ is given by a family $\cY \to \cM$, with fibres isomorphic to $(\C^*)^n$, endowed with a regular function
\begin{equation*}
W\!: \cY \to \C,
\end{equation*}
the Landau-Ginzburg (LG) potential. A (complexified) K\"ahler class $[\omega_0]$ determines a point in $\cM$ and so a corresponding regular function on the fibre, which we denote by $W(\omega_0)$.

Let $\Lambda_+$ denote the Novikov ring. There are a mirror map $\tau = \tau(y) \in H^*(X, \C) \otimes \C [\![\Lambda_+]\!][\![y]\!]$ and a $\C[z] [\![\Lambda_+]\!][\![y]\!]$-linear mirror isomorphism
\begin{equation*}
\Theta\!: \GM(W) := H^{n+1}(\Omega^{\bullet}_{\cY/\cM}\{z\}, z d + d W \wedge) \xrightarrow{\,\,\cong\,\,} H^*(X, \C) \otimes \C[z] [\![\Lambda_+]\!][\![y]\!] 
\end{equation*}
such that  $\Theta$ intertwines the Gauss-Manin connection $\nabla^{\GM}$ with the pullback $\tau^*\nabla^{\operatorname{D}}$ of the quantum connection $\nabla^{\operatorname{D}}$ by $\tau$ (see \cite{CoatesCortiIritani_hodge} Section 3.2 for the latter). 

In the $z\to 0$ limit, choosing a suitable generator for $\GM(W)$, namely $\Theta^{-1}_{\omega_0}(1)$, $\Theta$ specialises to an isomorphism between the log Jacobi ring and the quantum cohomology ring.

Moreover, the linear mirror isomorphism $\Theta$ intertwines a natural pairing on $\GM(W)$, known as the higher residue pairing $P$, with the Poincar\'e pairing.

The mirror map is in fact analytic in an open neighbourhood of the large volume limit point. Thus, for fixed, sufficiently large K\"ahler class $\beta$, we have a linear mirror isomorphism 
\begin{equation*}
\Theta_{\beta}\!: \GM(W(\beta)) \xrightarrow{\,\,\cong\,\,} H^*(X, \C) \otimes \C[\![z]\!] 
\end{equation*}
such that 
\begin{equation*}
P_{\beta}(\Omega_1, \Omega_2) = \left(\Theta_{\beta}(\Omega_1)\big|_{z \mapsto -z}, \Theta_{\beta}(\Omega_2)\right).
\end{equation*}
Note that the higher residue pairing $P$ has a description in terms of asymptotic expansions of integrals, explained in \cite{CoatesCortiIritani_hodge}, Section 6.2.
\begin{exm}\label{GiventalFormula} Suppose $X$ is toric Fano, with toric divisors $D_i$, $i = 1, \ldots, m$. Fix a toric fan $\Sigma$ for $X$. Then, the mirror map is trivial, and, in a fixed trivialisation of the mirror family, the LG potential is given by
\begin{equation*}
W = \sum^m_{i=1} a_i( \omega ) x^{v_i}, 
\end{equation*}
where $x_1, \ldots, x_n$ are torus coordinates, $v_i$ is the primitive generator of the ray of $\Sigma$ dual to $D_i$, and the $a_i$ are certain coefficients, uniquely determined, up to rescalings of the torus variables, by the condition that, for any integral linear relation
\begin{equation*}
\sum d_i v_i = 0,
\end{equation*}
corresponding to a unique curve class $[C]$ such that $D_i . [C] = d_i$, we have
\begin{equation*}
\prod_i a^{d_i}_i = e^{-2\pi \int_C\omega}. 
\end{equation*}
\end{exm}

\section{Toric weak Fano manifolds}
In this Section, after recalling Iritani's toric $\widehat{\Gamma}$-theorem, we provide the proofs of Theorem \ref{FanoThm} and Proposition \ref{LowerPhaseProp}.
\subsection{$\widehat{\Gamma}$-theorem}\label{GammaSec}
Let $X$ be a projective toric weak Fano manifold, with toric boundary $D = \sum^m_{i=1} D_i$. Fix a nef basis $p_1, \ldots, p_k$ of $\kk^{\vee}_{\Z} \cong H^2(X, \Z)$, with dual coordinates $q_1, \ldots, q_k$. Set
\begin{equation*}
q^{p/z} := e^{\sum^k_{i=1} p_i \log q_i/z}
\end{equation*}
and, for $d \in \kk_{\Z}$,
\begin{equation*}
q^d = \prod^k_{i = 1} q^{p_i \cdot d}_i.
\end{equation*}
Recall \emph{Givental's $I$-function}, with values in cohomology, is given by
\begin{equation*}
I(q, z) = \sum_{d \in A^{\vee}_X \cap \kk_{\Z}} q^{d + p/z} \prod^m_{i = 1}\frac{\prod^0_{j = -\infty}( D_i + j z)}{\prod^{D_i \cdot d}_{j = -\infty} (D_i + j z)},
\end{equation*}
where $A^{\vee}_X$ is the cone of curve classes. The relation of $I(q, z)$ with genus $0$ Gromov-Witten theory is discussed in \cite{Givental_toric}.

The \emph{Gamma class} in the toric case can be defined as
\begin{equation*}
\Gcl_X = \prod^m_{i = 1} \exp\left(-\gamma D_i + \sum^{\infty}_{k = 2} (-1)^k \frac{\zeta(k)}{k} D^k_i\right). 
\end{equation*}
\begin{thm}[Iritani \cite{Iritani_gamma}, \cite{Iritani_survey} Remark 11]\label{GammaThm} Let $X$ be a toric weak Fano manifold. Fix $z > 0$. Then, there exists an isomorphism 
\begin{equation*}
\Gamma\!: K^0(X) \xrightarrow{\sim} H_n(\cY_{q_k}, \{\Rea(W(k\omega_0)/z) \gg 0\}; \Z)
\end{equation*}
such that 
\begin{equation*}
\int_X \left(z^{c_1(X)} z^{\frac{\deg}{2}} I(q, -z)\right)\cup \Gcl_X (2\pi \ii)^{\frac{\deg}{2}} \ch(E) = \int_{\Gamma(E)} e^{-W_q/z} \Omega_0,
\end{equation*}
where $\deg \in \operatorname{End}(H^*(X, \C))$ denotes the grading operator.
\end{thm}
\begin{rmk} In this statement, for a fixed value of the parameter $z >0$, cycles in $H_n(\cY_{q_k}, \{\Rea(W(k\omega_0)/z) \gg 0\}; \Z)$ are normalised by a factor $(-2\pi z)^{-n/2}$, see \cite{Iritani_survey}, footnote 3. 
\end{rmk}
\subsection{Proof of Theorem \ref{FanoThm}}\label{FanoThmSec}
In this Section we mix additive and multiplicative notation for line bundles, and often suppress the tensor product in the notation. The proof involves naturally a parameter $z > 0$; the statement of Theorem \ref{FanoThm} is obtained by specialising $z = 1$.

Recall that the $\widehat{\Gamma}$-theorem (Theorem \ref{GammaThm}) involves the integral
\begin{equation}\label{GammaIntegral}
\int_X z^{c_1(X)} z^{\frac{\deg}{2}} I(q, -z)\cup \Gcl_X (2\pi \ii)^{\frac{\deg}{2}} \ch(E). 
\end{equation}
For our applications, we choose coordinates $q_i$ such that 
\begin{equation*}
\sum^k_{i=1} p_i \log q_i = 2\pi \ii [\beta_0 + \ii \omega_0],
\end{equation*}
so that, for $d \in \kk_{\Z}$, we have
\begin{equation*}
q^{d - p/z} = \prod^k_{i = 1} q^{p_i \cdot d}_i e^{-2\pi \ii [\beta_0 + \ii \omega_0]/z},
\end{equation*}
and the integral \eqref{GammaIntegral} can be expressed as
\begin{equation}\label{ExplicitGammaIntegral}
\int_X z^{c_1(X)} z^{\frac{\deg}{2}} \left(e^{-2\pi \ii [\beta_0 + \ii \omega_0]/z} \sum_{d \in A^{\vee}_X \cap \kk_{\Z}} q^{d} \prod^m_{i = 1}\frac{\prod^0_{j = -\infty}( D_i - j z)}{\prod^{D_i \cdot d}_{j = -\infty} (D_i - j z)}\right)\cup \Gcl_X (2\pi \ii)^{\frac{\deg}{2}} \ch(E).
\end{equation}

Fix a (large) positive integer $k$. Let us consider the integral \eqref{ExplicitGammaIntegral} with respect to the rescaled K\"ahler class $k [\omega_0]$, with vanishing $B$-field $[\beta_0] = 0$, and when $E = L^{k}$ is a power of a line bundle, namely
\begin{equation}\label{ScaledGammaIntegral} 
\int_X \left(z^{c_1(X)} z^{\frac{\deg}{2}} e^{2\pi k [\omega_0]/z} \sum_{d \in A^{\vee}_X \cap \kk_{\Z}} q^{k d} \prod^m_{i = 1}\frac{\prod^0_{j = -\infty}( D_i - j z)}{\prod^{D_i \cdot d}_{j = -\infty} (D_i - j z)}\right)\cup \Gcl_X (2\pi \ii)^{\frac{\deg}{2}} \ch(L^{k}).
\end{equation} 
Then, we have
\begin{equation*}
\sum_{d \in A^{\vee}_X \cap \kk_{\Z}} q^{k d} \prod^m_{i = 1}\frac{\prod^0_{j = -\infty}( D_i - j z)}{\prod^{D_i \cdot d}_{j = -\infty} (D_i - j z)} = 1 + O(k^{-1})
\end{equation*}
and \eqref{ScaledGammaIntegral} admits an expansion of the form
\begin{align*}  
& k^n \int_X e^{ 2\pi  [\omega_0]} (2\pi \ii)^{\frac{\deg}{2}}\ch(L) + O(k^{n-1})\\
&  = (2\pi \ii)^n k^n \int_X e^{2\pi [\omega_0]} \left(\frac{1}{2\pi \ii}\right)^{n-\frac{\deg}{2}}\ch(L) + O(k^{n-1})\\
&  = (2\pi \ii)^n k^n \int_X e^{\frac{2\pi}{2\pi \ii} [\omega_0]} \ch(L) + O(k^{n-1})\\
&  = (2\pi \ii)^n k^n \int_X e^{-\ii [\omega_0]} \ch(L) + O(k^{n-1}).
\end{align*}
Then, Theorem \ref{GammaThm} gives
\begin{equation}\label{AsympX}
\int_X e^{- \ii [\omega_0]} \ch(L) = (2\pi \ii)^{-n} k^{-n} \int_{\Gamma(L^{\otimes k})} e^{-W(k \omega_0)/z} \Omega_0 + O(k^{-1}).
\end{equation}
\begin{rmk}\label{AsympZRmk} Since \eqref{ScaledGammaIntegral} is of the form
\begin{equation*} 
\int_X e^{- \ii k [\omega_0]} \ch(L^{\otimes k})\,(1+O(k^{-1})),
\end{equation*}
\eqref{AsympX} also shows
\begin{equation*} 
\int_X e^{- \ii k [\omega_0]} \ch(L^{\otimes k}) = \frac{1}{(2\pi \ii)^n} \int_{\Gamma(L^{\otimes k})} e^{-W(k \omega_0)/z} \Omega_0\,(1 + O(k^{-1})).
\end{equation*}
\end{rmk}
\begin{rmk} In the special case when $L \cong \olo_X$, this is the leading term in the full asymptotic expansion
\begin{equation*}
\int_{(\R_{>0})^n} e^{-W(q)/z} \Omega_0 \sim \int_X q^{-p} z^{c_1(X)} \cup \Gcl_X
\end{equation*}
observed in \cite{Iritani_survey}, Section 7, equation (8).
\end{rmk}
\subsubsection{Divisors}\label{DivisorsSec}
Suppose that $V \subset X$ is a toric divisor. Consider the integral 
\begin{equation}\label{ScaledGammaIntegralV} 
\int_X  z^{c_1(X)} z^{\frac{\deg}{2}} \left(e^{- \ii k [\omega_0]/z}\sum_{d \in A^{\vee}_X \cap \kk_{\Z}} q^{k d} \prod^m_{i = 1}\frac{\prod^0_{j = -\infty}( D_i - j z)}{\prod^{D_i \cdot d}_{j = -\infty} (D_i - j z)}\right)\cup \Gcl_X (2\pi \ii)^{\frac{\deg}{2}} \ch(L^{k}(k_V V)),
\end{equation} 
where $k \gg k_V \gg 1$ are positive integers. Arguing as above, we see that \eqref{ScaledGammaIntegralV} admits an expansion of the form
\begin{align*}  
&(2\pi \ii )^{n} k^n \int_X e^{- \ii [\omega_0]} \ch(L)\\& + (2\pi \ii )^{n} k^{n-1} k_V \int_X e^{-\ii [\omega_0]} \ch(L) \cup c_1(\olo(V)) + O(k^{n-1}, k^0_V).
\end{align*}
At the same time, by Theorem \ref{GammaThm}, the integral \eqref{ScaledGammaIntegralV} equals
\begin{equation*}
\int_{\Gamma( L^{ k}(k_V V) )} e^{-W(k \omega_0)/z} \Omega_0 = \int_{\Gamma(L^{k})} e^{-W(k \omega_0)/z} \Omega_0 + \int_{\Gamma(\cS(k L, k_V V))} e^{-W(k \omega_0)/z} \Omega_0,
\end{equation*}
where we set 
\begin{equation*}
\cS(k L, k_V V) := L^{k}(k_V V)\otimes \olo_{k_V V},
\end{equation*}
and we use the exact sequence
\begin{equation}\label{DivisorExactSeq}
0 \to L^{k} \to L^{k}(k_V V) \to \cS(k L , k_V V) \to 0,
\end{equation}
together with the fact that $\Gamma$ is an isomorphism on $K$-theory. Combining this with \eqref{AsympX} shows
\begin{align}\label{AsympV}
\nonumber&\int_V e^{- \ii  [\omega_0]} \ch(L)\\
\nonumber&=\int_X e^{- \ii  [\omega_0]} \ch(L) \cup c_1(\olo(V))\\
& = (2\pi \ii)^{-n} k^{- n + 1} k^{-1}_V \int_{\Gamma(\cS(k L, k_V V))} e^{-W(k \omega_0)/z} \Omega_0 + O(k^{-1}_V).
\end{align}

Then, the above argument, in particular \eqref{AsympX} and \eqref{AsympV}, show that the phase inequality \eqref{PhaseIneq} (which is equivalent to dHYM positivity \eqref{dHYMPosIntro}), namely
\begin{equation}\label{PhaseIneqDivisor} 
\arg\left((-1)^{\codim V} \int_V e^{-\ii \omega_0}\ch(L)\right) < \arg\left(\int_X e^{-\ii \omega_0}\ch(L)\right),
\end{equation}
implies
\begin{equation}\label{PeriodIneqDivisor}
\arg\left((-1)^{\codim V}\int_{\Gamma(\cS(k L, k_V V))} e^{-W(k \omega_0)/z} \Omega_0 \right) <  \arg \int_{\Gamma(L^{\otimes k})} e^{-W(k \omega_0)/z} \Omega_0,
\end{equation}
for all $k \gg k_V \gg 1$. 

Conversely, if we assume the inequality \eqref{PeriodIneqDivisor}, then we get the inequality of \emph{leading order terms} in the expansion for $k \gg k_V \gg 1$, 
\begin{equation*}  
\arg\left((-1)^{\codim V} \int_V e^{-\ii \omega_0}\ch(L)\right) \leq \arg\left(\int_X e^{-\ii \omega_0}\ch(L)\right).
\end{equation*}

In general, the inequality may not be strict. However, this is ruled out if we suppose that $[\omega_0]$ is \emph{generic with respect to $V$} in the sense that it does not lie in the proper analytic subvariety of $H^{1, 1}(X, \R)$ determined by $c_1(L)$ and $V$ through the condition
\begin{equation*}  
\arg\left((-1)^{\codim V} \int_V e^{-\ii \omega_0}\ch(L)\right) = \arg\left(\int_X e^{-\ii \omega_0}\ch(L)\right),
\end{equation*}
which, under our assumptions, is precisely the condition that $V$ is strictly semi-stabilising in the Nakai-Moishezon criterion, 
\begin{equation}\label{SemiStabCondition}
\int_V \Rea(\ii \omega_0 + c_1(L^{\vee}))^{\dim V} - \cot(\varphi)\Imm(\ii \omega_0 +  c_1(L^{\vee}))^{\dim V} = 0.
\end{equation}

The argument above holds for a fixed toric divisor $V$. However, since there are only finitely many toric divisors, it is clear that the parameters $k \gg k_1 \gg 1$ appearing in \eqref{PeriodIneqDivisor} can be chosen uniformly over all $V$. Moreover, $k \gg k_1 \gg 1$ can be chosen \emph{uniformly} over all $L$ whose $K$-theory class lies in a bounded subset of $K^0(X)$ (in any norm). 

Then the equivalence between the conditions \eqref{PhaseIneqDivisor} and \eqref{PeriodIneqDivisor} holds as long as $[\omega_0]$ is \emph{generic with respect to all divisors} in the sense that it does not lie in the union of finitely many proper analytic subvarieties of $H^{1, 1}(X, \R)$ determined by $c_1(L)$ and the classes of all toric divisors.

Finally, we note that the exact sequence \eqref{DivisorExactSeq} induces the boundary morphism in $D^b(X)$
\begin{equation*}
\cS(k L, k_V V) \to L^{\otimes k}[1],   
\end{equation*}
identified with the map of complexes
\begin{equation*}
\begin{tikzcd}
     L^{\otimes k} \arrow{r}\arrow{d} & L^{\otimes k }(k_V V)\arrow{d}\\
     L^{\otimes k} \arrow{r} & 0.  
\end{tikzcd}
\end{equation*}
Thus, the phase inequalities \eqref{PhaseIneqDivisor} and \eqref{PeriodIneqDivisor} corresponds to a morphism 
\begin{equation*}
\cS(k L, k_V V)[-1] \to L^{\otimes k}.
\end{equation*} 
\begin{rmk}\label{CentralChargeRmk} As observed in \cite{CollinsYau_dHYM_momentmap}, Section 8, there is a basic difficulty in comparing dHYM positivity \eqref{dHYMPosIntro} with algebro-geometric stability conditions, due to the fact that, in general,
\begin{equation*}
 \int_V e^{-\ii \omega_0}\ch(L) \neq \int_X e^{-\ii \omega_0}\ch(L \otimes \olo_V).
\end{equation*}
From our current viewpoint, this is replaced by the approximate equality
\begin{align*}
\int_V e^{-\ii \omega_0}\ch(L) = k^{- n + 1} k^{-1}_V \int_X e^{-\ii k \omega_0}\ch(\cS(k L, k_V V)) + O(k^{-1}_V),
\end{align*}
which follows from our computations above. Thus, the phase inequality \eqref{PhaseIneq} (which is equivalent to dHYM positivity \eqref{dHYMPosIntro}), namely
\begin{equation*}  
\arg\left((-1)^{\codim V} \int_V e^{-\ii \omega_0}\ch(L)\right) < \arg\left(\int_X e^{-\ii \omega_0}\ch(L)\right),
\end{equation*}
can be expressed as
\begin{equation*}
\arg\left( (-1)^{\codim V}\int_{X} e^{-\ii k\omega_0}\ch(\cS(k L, k_V V))\right)< \arg \int_{X} e^{-\ii k\omega_0}\ch(L^{\otimes k}), 
\end{equation*}
where we have a morphism $\cS(k L, k_V V)[-\codim V] \to L^{\otimes k}$ in $D^b(X)$, provided that $[\omega_0]$ is \emph{generic with respect to all divisors} in the sense that it does not lie in the union of finitely many proper analytic subvarieties of $H^{1, 1}(X, \R)$ determined by $c_1(L)$ and the classes of all toric divisors through the conditions \eqref{SemiStabCondition} (as $V$ varies). 

We will use this fact in the proof of Theorem \ref{BstabThm}, concerning Bridgeland stability conditions for Lagrangian sections, when $X$ is a toric weak del Pezzo surface.
\end{rmk}
\subsubsection{Complete intersections}
The general case when the toric submanifold $V$ is a complete intersection of toric divisors can be obtained by induction. We illustrate first the case 
\begin{equation*}
V := V_1 \pitchfork V_2 
\end{equation*}
for the sake of clarity. 

We apply the identity \eqref{AsympV}, replacing the line bundle $L$ with $L^{k}(k_1 V_1)$, the divisor $V$ with $V_2$, and the K\"ahler form $\omega_0$ with $k \omega_0$. This gives
\begin{align*} 
\nonumber&\int_{V_2} e^{-  \ii   k [\omega_0]} \ch(L^{k}(k_1 V_1))\\
\nonumber&=\int_{X} e^{- \ii   k [\omega_0]} \ch(L^{k}(k_1 V_1))\cup c_1(\olo(V_2))\\
& = (2\pi \ii)^{-n} \tilde{k}^{- n + 1} k^{-1}_2 \int_{\Gamma(\cS(\tilde{k}L^{k}(k_1 V_1), k_2 V_2))} e^{-W(k \tilde{k} \omega_0)/z} \Omega_0 + O(k^{-1}_2).
\end{align*} 
At the same time, we have an expansion
\begin{align*}
& \int_{V_2} e^{- \ii   k [\omega_0]} \ch(L^{k}(k_1 V_1))\\
& = k^{n-1} \int_{V_2} e^{- \ii   [\omega_0]} \ch(L) + k^{n-2} k_1 \int_{V_2} e^{- \ii z [\omega_0]} \ch(L) \cup c_1(V_1) + O(k^{n-2}, k^0_1)\\
& = k^{n-1} \int_{V_2} e^{- \ii   [\omega_0]} \ch(L) + k^{n-2} k_1 \int_{V_1 \pitchfork V_2} e^{- \ii z [\omega_0]} \ch(L)  + O(k^{n-2}, k^0_1). 
\end{align*}
Using \eqref{AsympV} once again shows that, for positive integers $h$, we have 
\begin{align*} 
\nonumber&\int_{V_2} e^{- \ii   [\omega_0]} \ch(L)\\
& = (2\pi \ii)^{-n} (k \tilde{k})^{- n + 1} h^{-1} \int_{\Gamma(Q( k \tilde{k} L, h V_2))} e^{-W((k \tilde{k})\omega_0)/z} \Omega_0 + O(h^{-1}),
\end{align*}
which gives
\begin{align*}
\int_{V_2} e^{- \ii   k [\omega_0]} \ch(L^{k}(k_1 V_1)) &= (2\pi \ii)^{-n} \tilde{k}^{- n + 1} h^{-1} \int_{\Gamma(Q( k \tilde{k} L, h V_2))} e^{-W((k \tilde{k})\omega_0)/z} \Omega_0\\ 
&+ k^{n-2} k_1 \int_{V_1 \pitchfork V_2} e^{- \ii   [\omega_0]} \ch(L)  + O(k^{n-2}, k^0_1) + O(h^{-1}). 
\end{align*}
So we find
\begin{align}
\nonumber &\int_{V_1 \pitchfork V_2} e^{- \ii   [\omega_0]} \ch(L)\\
\nonumber &= (2\pi \ii)^{-n} k^{-n+2} \tilde{k}^{- n + 1} k^{-1}_1 k^{-1}_2 \int_{\Gamma(Q(\tilde{k}L^{k}(k_1 V_1), k_2 V_2))} e^{-W(k \tilde{k} \omega_0)/z} \Omega_0\\
\nonumber & - (2\pi \ii)^{-n} k^{-n+2}  \tilde{k}^{- n + 1} k^{-1}_1 h^{-1} \int_{\Gamma(Q( k \tilde{k} L, h V_2))} e^{-W((k \tilde{k})\omega_0)/z} \Omega_0 \\
&  + O(k^{-n+ 2} k^{-1}_1 h^{-1}) + O(k^{-n+2} k^{-1}_1 k^{-1}_2) + O( k^{-1}_1).
\end{align}\label{AsympManyVMessy}
We regard $\cS(\tilde{k}L^{k}(k_1 V_1), k_2 V_2)$, $\cS( k \tilde{k} L, h V_2))$ as elements in $D^b(X)$, isomorphic to the complexes
\begin{align*}
&\cS(\tilde{k}L^{k}(k_1 V_1), k_2 V_2) \cong [L^{ k \tilde{k}}(\tilde{k} k_1 V_1) \to L^{  k \tilde{k}}(\tilde{k} k_1 V_1 + k_2 V_2)],\\ 
&\cS( k \tilde{k} L, h V_2)) \cong [L^{ k \tilde{k}} \to L^{ k \tilde{k}}(h V_2)]. 
\end{align*}
Making the natural choice
\begin{equation*}
h = k_2
\end{equation*}
yields a morphism $\cS( k \tilde{k} L, k_2 V_2)) \to \cS(\tilde{k}L^{k}(k_1 V_1), k_2 V_2)$ induced by the map of complexes
\begin{equation*}
\begin{tikzcd}
    L^{\otimes k \tilde{k}} \arrow{r} \arrow{d} & L^{\otimes k \tilde{k}}(k_2 V_2) \arrow{d}  \\
    L^{\otimes k \tilde{k}}(\tilde{k} k_1 V_1) \arrow{r} & L^{\otimes k \tilde{k}}(\tilde{k} k_1 V_1 + k_2 V_2).
  \end{tikzcd}
\end{equation*}
We define 
\begin{equation*}
\cS(k \tilde{k} L, \tilde{k} L^{k}(k_1 V_1), k_2 V_2) \in D^b(X)
\end{equation*}
as the cone
\begin{equation*}
\cS( k \tilde{k} L, k_2 V_2)) \to \cS(\tilde{k} L^{k}(k_1 V_1), k_2 V_2)) \to \cS(k \tilde{k} L, \tilde{k} L^{k}(k_1 V_1), k_2 V_2) \to \cS( k \tilde{k} L, k_2 V_2))[1],
\end{equation*}
and we have in $K$-theory
\begin{equation*}
[\cS(k \tilde{k} L, \tilde{k} L^{k}(k_1 V_1), k_2 V_2)] = [\cS(\tilde{k} L^{k}(k_1 V_1), k_2 V_2))] - [\cS( k \tilde{k} L, k_2 V_2))].
\end{equation*}
With our choice $h = k_2$, and setting
\begin{equation*}
\hat{h} := k \tilde{k},
\end{equation*}
we can rewrite the expansion \eqref{AsympManyVMessy} as
\begin{align}\label{AsympManyV} 
\nonumber &\int_{V_1 \pitchfork V_2} e^{- \ii   [\omega_0]} \ch(L)\\
\nonumber &= (2\pi \ii)^{-n} k {\hat{k}}^{-n+1} k^{-1}_1 k^{-1}_2 \int_{\Gamma(\cS(\hat{k} L, k^{-1}\hat{k} L^{k}(k_1 V_1), k_2 V_2))} e^{-W(\hat{k} \omega_0)/z} \Omega_0\\ 
&  + O(k^{-n+ 2} k^{-1}_1 k^{-1}_2) + O( k^{-1}_1).
\end{align}

Suppose that $[\omega_0]$ is \emph{generic with respect to  $V_1$, $V_2$} in the sense that it does not lie in the uion of finitely many proper analytic subvarieties of $H^{1, 1}(X, \R)$ determined by $c_1(L)$, $V_1$, $V_2$.

Then, by \eqref{AsympX}, the phase inequality \eqref{PhaseIneq} in this case  
\begin{equation*}
\arg\left((-1)^{ \codim V_1 \pitchfork V_2 } \int_{V_1 \pitchfork V_2} e^{- \ii   [\omega_0]} \ch(L)\right) < \arg \int_{X} e^{- \ii   [\omega_0]} \ch(L) 
\end{equation*}
is equivalent to 
\begin{equation*}
\arg\left((-1)^{ \codim V_1 \pitchfork V_2 }\int_{\Gamma(\cS(\hat{k} L, k^{-1}\hat{k} L^{k}(k_1 V_1), k_2 V_2))} e^{-W(k \tilde{k} \omega_0)/z} \Omega_0\right) < \int_{\Gamma(L^{\otimes  k \tilde{k} })} e^{-W(k \tilde{k} \omega_0)/z} \Omega_0. 
\end{equation*}
We can compose the (shifted) boundary morphisms
\begin{equation*}
\cS(\hat{k} L, k^{-1}\hat{k} L^{k}(k_1 V_1), k_2 V_2) \to \cS( \hat{k} L, k_2 V_2))[1]
\end{equation*}
and
\begin{equation*}
\cS( \hat{k} L, k_2 V_2))[1] \to L^{\otimes \hat{k}}[2]
\end{equation*}
to obtain a morphism
\begin{equation*}
\cS(\hat{k} L, k^{-1}\hat{k} L^{k}(k_1 V_1), k_2 V_2) \to L^{\otimes \hat{k}}[2].
\end{equation*}

The upshot of our discussion it that, for all sufficiently large $k$, independent of the choice of toric divisors $V$, $V_1$, $V_2$, there exist cycles 
\begin{equation*}
\Gamma_V,\,\Gamma_{V_1 \pitchfork V_2} \in H_n(\cY_{q_k}, \{\Rea(W(k\omega_0)/z) \gg 0\}; \Z)
\end{equation*}
such that dHYM positivity \eqref{dHYMPosIntro} with respect to all $V$ and $V_1 \pitchfork V_2$ is equivalent to
\begin{align*}
& \arg\left((-1)^{ \codim V }\int_{\Gamma_V} e^{-W(k \omega_0)/z} \Omega_0\right) < \int_{\Gamma(L^{\otimes  k})} e^{-W(k \omega_0)/z} \Omega_0,\\
& \arg\left((-1)^{ \codim V_1 \pitchfork V_2 } \int_{\Gamma(V_1\pitchfork V_2)} e^{-W(k \omega_0)/z} \Omega_0 \right) < \int_{\Gamma(L^{\otimes  k})} e^{-W(k \omega_0)/z} \Omega_0. 
\end{align*}
Note that we chose $k$ larger to include intersections $V_1 \pitchfork V_2$. Moreover, by construction, we have
\begin{equation*}
\Gamma_V = \Gamma(\cS_V),\,\Gamma_{V_1 \pitchfork V_2} = \Gamma(\cS_{V_1\pitchfork V_2})
\end{equation*} 
for suitable objects 
\begin{equation*}
\cS_V,\, \cS_{V_1\pitchfork V_2} \in D^b(X),
\end{equation*}
with natural morphisms
\begin{equation*}
 \cS_V[-1] \to L^{\otimes k},\, \cS_{V_1\pitchfork V_2}[-2] \to L^{\otimes k}.
\end{equation*}
It is also clear that $k$ can also be chosen uniformly as long as $[L]$ varies in a bounded set on $K^0(X)$. 

Finally, for the induction step, we note that if $V_1 \pitchfork V_2 \pitchfork V_3$ is a complete intersection of toric divisors, then, for all $k, k_i > 0$, there is a natural morphism
\begin{align*}
\cS(k\hat{k} L, \hat{k} L^{k}(k_1 V_1), k_2 V_2) \to \cS(\hat{k} L^{k}(k_1 V_1), \hat{k} L^{k}(k_1 V_1)(k^{-1} k_2 V_2), k_3 V_3)
\end{align*}
induced by the diagram
\begin{equation*}
\begin{tikzcd}
    \cS( k \hat{k} L, k_2 V_2)) \arrow{r} \arrow{d} & \cS(\hat{k} L^{k}(k_1 V_1), k_2 V_2)) \arrow{d}  \\
    \cS(\hat{k} L^{k}(k_1 V_1), k_3 V_3) \arrow{r} & \cS(\hat{k} L^{k}(k_1 V_1)(k^{-1} k_2 V_2), k_3 V_3) .
  \end{tikzcd}
\end{equation*}
\subsubsection{Induction}
Consider a general complete intersection of toric divisors 
\begin{equation*}
V = V_1 \pitchfork \cdots\pitchfork V_r.
\end{equation*}
We can assume inductively, starting from \eqref{AsympV} (or \eqref{AsympManyV}), that there exists an object 
\begin{align*}
& \cS(\hat{k} L^{k}(k_1 V_1), k^{-1}\hat{k} (L^{k}(k_1 V_1))^{k}(k_2 V_2), \ldots, k_r V_r) \\
& = \cS(\hat{k} L^{k}(k_1 V_1), \hat{k} L^{k}(k_1 V_1)(k^{-1} k_2 V_2), \ldots, k_r V_r) \in D^b(X)
\end{align*}
such that we have an expansion
\begin{align*} 
\nonumber &\int_{V_2 \pitchfork \cdots \pitchfork V_r} e^{- \ii   k [\omega_0]} \ch(L^{k}(k_1 V_1))\\
\nonumber &= (2\pi \ii)^{-n} k^{r-1} {\hat{k}}^{-n+1} \prod^r_{i = 2} k_i \int_{\Gamma(\cS(\hat{k} L^{k}(k_1 V_1), \hat{k} L^{k}(k_1 V_1)(k^{-1} k_2 V_2), \ldots, k_r V_r))} e^{-W(k \hat{k} \omega_0)/z} \Omega_0\\ 
&  + O(k^{-n+ r - 1}, k^{-1}_2, \ldots, k^{-1}_r).
\end{align*}
At the same time, we have
\begin{align*}
& \int_{V_2 \pitchfork \cdots \pitchfork V_r} e^{- \ii   k [\omega_0]} \ch(L^{k}(k_1 V_1))\\
& = k^{n-r + 1} \int_{V_2 \pitchfork \cdots \pitchfork V_r} e^{- \ii   [\omega_0]} \ch(L) + k^{n- r} k_1 \int_{V_2 \pitchfork \cdots \pitchfork V_r} e^{- \ii   [\omega_0]} \ch(L) \cup c_1(V_1) + O(k^{n-r}, k^0_1)\\
& = k^{n-r + 1} \int_{V_2 \pitchfork \cdots \pitchfork V_r} e^{- \ii   [\omega_0]} \ch(L) + k^{n- r} k_1 \int_{V_1 \pitchfork \cdots \pitchfork V_r} e^{- \ii   [\omega_0]} \ch(L) + O(k^{n-r}, k^0_1). 
\end{align*}
By induction from \eqref{AsympV} or \eqref{AsympManyV}, we also know that there exists 
\begin{align*}
& \cS(k\hat{k} L, \hat{k} L^{k}(k_2 V_2), \ldots, k_r V_r) \in D^b(X)
\end{align*}
such that
\begin{align*} 
\int_{V_2 \pitchfork \cdots \pitchfork V_r} e^{- \ii   [\omega_0]} \ch(L) &= (2\pi \ii)^{-n} k^{r-1} {\hat{k}}^{-n+1} \prod^r_{i = 2} k_i \int_{\Gamma(\cS(k\hat{k} L, \hat{k} L^{k}(k_2 V_2), \ldots, k_r V_r))} e^{-W(k\hat{k} \omega_0)/z} \Omega_0\\ 
&  + O(k^{-n+ r - 1}, k^{-1}_2,\ldots, k^{-1}_r).
\end{align*}
Moreover, inductively, we know that there is a morphism in $D^b(X)$
\begin{align*}
\cS(k\hat{k} L, \hat{k} L^{k}(k_2 V_2), \ldots, k_r V_r) \to \cS(\hat{k} L^{k}(k_1 V_1), \hat{k} L^{k}(k_1 V_1)(k^{-1} k_2 V_2), \ldots, k_r V_r)
\end{align*}
and we can define
\begin{align*}
&\cS_{V_1 \pitchfork \cdots \pitchfork V_r} := \cS(k\hat{k} L, \hat{k} L^{k}(k_1 V_1), \hat{k} L^{k}(k_1 V_1)(k^{-1} k_2 V_2), \ldots, k_r V_r)\in D^b(X)
\end{align*}
as its cone. Then, we find 
\begin{align*}
\int_{V_1 \pitchfork \cdots \pitchfork V_r} e^{- \ii   [\omega_0]} \ch(L) &= - (2\pi \ii)^{-n} k^{r } {\hat{k}}^{-n+1} \prod^r_{i = 1} k_i \int_{\Gamma(\cS_{V_1 \pitchfork \cdots \pitchfork V_r})} e^{-W(k\hat{k} \omega_0)/z} \Omega_0\\
&+O(k^{-n + r}, k^{-1}_1, \ldots, k^{-1}_r).
\end{align*}
The boundary map  
\begin{equation*}
\cS_{V_1 \pitchfork \cdots \pitchfork V_r} \to \cS(k\hat{k} L, \hat{k} L^{k}(k_2 V_2), \ldots, k_r V_r)[1] 
\end{equation*}
can be composed with the inductively given morphism
\begin{equation*}
\cS(k\hat{k} L, \hat{k} L^{k}(k_2 V_2), \ldots, k_r V_r)[1] \to L^{\otimes k\hat{k}}[r]
\end{equation*}
in order to obtain
\begin{equation*}
\cS_{V_1 \pitchfork \cdots \pitchfork V_r} \to  L^{\otimes k\hat{k}}[r].
\end{equation*}

Suppose that $[\omega_0]$ is \emph{generic} in the sense that it does not lie in the union of finitely many proper analytic subvarieties of $H^{1, 1}(X, \R)$ determined by $c_1(L)$ and the classes of all toric submanifolds.

We can then choose our constants
\begin{equation*}
k,\, \hat{k} \gg k_1, \ldots, k_r \gg 1
\end{equation*}
\emph{uniformly}, so that, setting $\cS_V := \cS_{V_1 \pitchfork \cdots \pitchfork V_r}$, the phase inequality \eqref{PhaseIneq} is equivalent to 
\begin{equation*}
\arg\left((-1)^{ \codim V } \int_{\Gamma(\cS_{V})} e^{-W(k\hat{k} \omega_0)/z} \Omega_0 \right) < \arg \int_{\Gamma(L^{\otimes  k \hat{k}})} e^{-W(k \hat{k}\omega_0)/z} \Omega_0,
\end{equation*}
corresponding to the morphism
\begin{equation*}
\cS_{V}[-\codim V] \to  L^{\otimes k\hat{k}}.
\end{equation*}
\section{Relation to Bridgeland stability on surfaces}\label{BSec}
In this Section we study Theorem \ref{FanoThm} and its Corollaries \ref{FSCor}, \ref{ThomasYauCor} from the viewpoint of Bridgeland stability, when $X$ is a weak del Pezzo surface. In particular we prove Theorem \ref{BstabThm}.

Fix a K\"ahler class $[\omega_0]$. There is well-known Bridgeland stability condition $\sigma_{[\omega_0]}$ on $D^b(X)$, 
\begin{equation*}
\sigma_{[\omega_0]} = \left(\Coh^{\sharp}(X), Z_{[\omega_0]}\right),
\end{equation*}
with central charge given by 
\begin{equation*}
Z_{[\omega_0]}(E) = -\int_X e^{-\ii [\omega_0]}\ch(E), 
\end{equation*} 
for $E \in \Coh(X)$ (see e.g. \cite{Collins_stability}, Section 4). 

Let us denote by $\mu_{[\omega_0]}(E)$ the Mumford-Takemoto slope of a coherent sheaf. Then $\sigma_{[\omega_0]}$ is supported on the heart $\Coh^{\sharp}(X)$, depending on $[\omega_0]$, which is the tilting of $\Coh(X)$ at the torsion pair $(\Coh^{>0}(X), \Coh^{\leq 0}(X))$, where $\Coh^{>0}(X)$ is generated (by extensions) by slope semistable sheaves of slope $\mu_{[\omega_0]} > 0$, and by zero- or one-dimensional torsion sheaves, while $\Coh^{\leq 0}(X)$ is generated by semistable sheaves of slope $\mu_{[\omega_0]} \leq 0$. 

It follows that $\Coh^{\sharp}(X)$ has a cohomological description
\begin{align*}
&\Coh^{\sharp}(X)\\
& = \{E \in D^b(\Coh(X))\,:\, H^{0}(E) \in \Coh^{> 0},\,H^{-1}(E) \in \Coh^{\leq 0},\,H^{i}(E) = 0, i \neq -1, 0\}.
\end{align*}

Suppose that we are in the situation of Theorem \ref{FanoThm}, on a weak del Pezzo $X$. In the proof of Theorem \ref{FanoThm}, Section \ref{DivisorsSec} and Remark \ref{CentralChargeRmk}, we constructed objects $\cS(k L, k_V V) \in D^b(X)$, with morphisms 
\begin{equation}\label{BoundaryMorDiv}
\cS(k L, k_V V) \to L^{k}[1],
\end{equation}
where $V$ ranges through toric divisors of $X$, such that, for $k \gg k_V \gg 1$, dHYM positivity of $L^{\vee}$ with respect to $[\omega_0]$ is equivalent to the phase inequality of central charges 
\begin{equation}\label{BStabZPhaseIneq}
\arg Z_{k[\omega_0]}\left(\cS(k L, k_V V)\right) < \arg Z_{k[\omega_0]}\left(L^{k}[1]\right).
\end{equation}
Moreover, $\cS(k L, k_V V)$ is given explicitly by the complex
\begin{equation*}
[L^k \to L^k(k_V V)] \in D^b(X).
\end{equation*}
Assume now that $L^{\vee}$ is ample. Then, we have
\begin{equation*}
H^{-1}(L^{k}[1]) = L^k \in \Coh^{\leq 0}(X),\,H^i(L^{k}[1]) = 0, i \neq -1, 0, 
\end{equation*}
since $L^k$ is a negative line bundle and so $\mu_{[\omega_0]} \leq 0$ for all $[\omega_0]$. On the other hand, 
\begin{equation*}
H^{0}(\cS(k L, k_V V)) = L^{k}(k_V V)\otimes \olo_{k_V V} \in \Coh^{> 0}(X),\,H^i(\cS(k L, k_V V)) = 0, i \neq 0, 
\end{equation*}
by the exact sequence \eqref{DivisorExactSeq} and since $L^{k}(k_V V)\otimes \olo_{k_V V}$ is a one-dimensional torsion sheaf. 

This argument shows that, when $L^{\vee}$ is positive, the morphism \eqref{BoundaryMorDiv} is actually a morphism in the abelian category $\Coh^{\sharp}(X)$.

For our applications we also need to recall a conjecture of Arcara and Miles.
\begin{conj}[\cite{ArcaraMiles}, Conjecture 1]\label{AMConj} Let $X$ be a smooth projective surface with a fixed K\"ahler class $[\omega_0]$ and a line bundle $L$. Then, 
\begin{enumerate}
\item[$(i)$] if $L$ is unstable with respect to $\sigma_{[\omega_0]}$, then it is destabilised by $L(-C)$, where $C$ is a curve of negatibe self-intersection;  
\item[$(ii)$] if $L[1]$ is unstable with respect to $\sigma_{[\omega_0]}$, then it is destabilised by $L(C)|_C$, where $C$ is a curve of negatibe self-intersection. 
\end{enumerate} 
\end{conj} 
Note that \cite{ArcaraMiles}, Conjecture 1 actually allows more general (divisorial) stability conditions. The conjecture is known for $X = \Bl_p \PP^2$ (see \cite{ArcaraMiles}) and for $X = \Bl_{p, q} \PP^2$ (see the recent work \cite{MizunoYoshida}).
\begin{prop}\label{BstabProp} Let $X$ be a toric surface. Fix a K\"ahler class $[\omega_0]$ and a positive line bundle $L^{\vee}$. Fix $k > 0$ sufficiently large. Suppose $[\omega_0]$ is \emph{generic} in the sense that it does not lie in the union of finitely many proper analytic subvarieties of $H^{1, 1}(X, \R)$ determined by $c_1(L)$. Then, the dHYM equation is solvable on $L^{\vee}$ with respect to $[\omega_0]$ if, and only if, the object $L^k[1]$ is Bridgeland stable with respect to the stability condition 
\begin{equation*}
\sigma_{k[\omega_0]} := \left(\Coh^{\sharp}(X), Z_{k[\omega_0]}\right).
\end{equation*}
The converse also holds, conditionally on the Arcara-Miles conjecture \ref{AMConj}, which is known when $X = \Bl_p \PP^2$ or $X = \Bl_{p, q} \PP^2$.
\end{prop}
\begin{proof} Suppose that $L^k[1]$ is stable with respect to $\sigma_{k[\omega_0]}$. Since we are assuming that $L^{\vee}$ is positive, the morphisms \eqref{BoundaryMorDiv} induce exact triangles of objects contained in the heart $\Coh^{\sharp}(X)$. Thus, by the definition of Bridgeland stability, the phase inequalities for central charges \eqref{BStabZPhaseIneq} must hold for all toric divisors $V$. Therefore, under our assumptions, by the proof of Theorem \ref{FanoThm}, Section \ref{DivisorsSec} and Remark \ref{CentralChargeRmk}, the line bundle $L^k$ is dHYM positive with respect to $k[\omega_0]$. So the same holds for $L$ with respect to $[\omega_0]$, and the dHYM equation with respect to $[\omega_0]$ is solvable on $L$ and $L^{\vee}$.

Conversely, suppose that the dHYM equation is solvable on the positive line bundle $L^{\vee}$ with respect to $[\omega_0]$. Thus, it is also solvable on the negative line bundle $L$. Following Collins-Shi \cite{Collins_stability}, proof of Theorem 4.15, we write the corresponding dHYM positivity condition as 
\begin{equation*}
C \cdot \ch_1(L) > \frac{\ch_2(L) - \frac{1}{2}[\omega_0]^2}{\ch_1(L)\cdot [\omega_0]} (C\cdot [\omega_0]),
\end{equation*}
for \emph{all} curves $C \subset X$. As we know, this must be invariant under our rescaling, i.e. equivalent to 
\begin{equation*}
C \cdot \ch_1(L^{ k}) > \frac{\ch_2(L^{ k}) - \frac{1}{2}[k\omega_0]^2}{\ch_1(L^{ k})\cdot [k\omega_0]} (C\cdot [k\omega_0]).
\end{equation*}

On the other hand, assuming the Arcara-Miles conjecture \ref{AMConj}, $L^{ k}$ is only potentially destabilised with respect to $\sigma_{k[\omega_0]}$ by the objects $L^{ k}(-C)$ where $C$ is a curve with $C^2 < 0$. The condition that such objects do \emph{not} destabilise is given by 
\begin{equation*}
\frac{\ch_2(L^{  k}) - C \cdot \ch_1(L^{ k}) + \frac{1}{2} C^2 -\frac{1}{2}[k\omega_0]^2}{\ch_1(L^{ k})\cdot[k\omega_0] - C\cdot[k\omega_0]} < \frac{\ch_2(L^{ k}) -\frac{1}{2}[k\omega_0]^2}{\ch_1(L^{ k})\cdot[k\omega_0]}.
\end{equation*}
Following again \cite{Collins_stability}, proof of Theorem 4.15, we note that this is equivalent to the condition
\begin{equation*}
C \cdot \ch_1(L^{ k}) - \frac{1}{2} C^2 > \frac{\ch_2( L^{ k}) -\frac{1}{2}[k\omega_0]^2}{\ch_1(L^{ k})\cdot[k\omega_0]}(C\cdot [k\omega_0]).
\end{equation*}
Since $C^2 < 0$, this shows that if the dHYM equation is solvable on $L^{\vee}$, then the line bundle $L^{k}$ is stable with respect to $\sigma_{k[\omega_0]}$. 

Finally, since $L$ is negative, the object $L^{k}[1]$ lies in the heart $\Coh^{\sharp}(X)$ of $\sigma_{k[\omega_0]}$. Thus it is stable with respect to $\sigma_{k[\omega_0]}$ iff its shift $L^{k}$ is, by the definition of Bridgeland stability condition. It follows that if the dHYM equation is solvable on $L^{\vee}$ then $L^{k}[1]$ is stable with respect to $\sigma_{k[\omega_0]}$.  
\end{proof}
\begin{exm}\label{CollinsShiExm} On $X = \Bl_p \PP^2$, write $h$, $e$ for the classes of the pullback of a line in $\PP^2$ and of the exceptional divisor. In the counterexample of Collins and Shi discussed in Remark \ref{CollinsShiExampleRmk}, one chooses 
\begin{equation*}
[\omega_0] = \frac{1}{\sqrt{3}} (2 h - e),\,c_1(L) = 2 h.  
\end{equation*} 
Note that $L$ is positive, so does not fit our conventions, although probably the example can be modified. 

According to \cite{Collins_stability}, Section 4, $L$ is Bridgeland stable with respect to $\sigma_{[\omega_0]}$, but it is not dHYM positive with respect to $[\omega_0]$, so the dHYM equation is not solvable on $L$ and $L^{\vee}$. 

However, we claim that $L^{k}$ is unstable with respect to $\sigma_{k[\omega_0]}$ for $k \gg 1$. Since $\mu_{[k\omega_0]}(L^k) > 0$, $L^k$ is an object of $\Coh^0(X)$, and similarly, denoting by $E$ the exceptional divisor, 
\begin{equation*}
\mu_{[k\omega_0]}(L^k(-E)) = \frac{(k \ch_1(L) - e)\cdot [k\omega_0]}{[k\omega_0]^2} = \frac{4 k^2 - k}{k^2 \sqrt{3}} > 0 
\end{equation*}
for $k \geq 1$. Thus, the morphism $L^k(-E) \to L^k$ is defined in $\Coh^0(X)$. The corresponding phase inequality   
\begin{equation*}
\arg Z_{[k\omega_0]}(L^k(-E)) < \arg Z_{[k\omega_0]}(L^k) 
\end{equation*}
is given by
\begin{equation*}
\frac{\ch_2(L^k) - e \cdot \ch_1(L^k) + \frac{1}{2} e^2 -\frac{1}{2}[k\omega_0]^2}{\ch_1(L^k)\cdot[k\omega_0] - e\cdot[k\omega_0]} < \frac{\ch_2(L^k) -\frac{1}{2}[k\omega_0]^2}{\ch_1(L^k)\cdot[k\omega_0]}.
\end{equation*}
By explicit computation, this is equivalent to 
\begin{equation*}
\frac{2 k^2 - \frac{1}{2} -\frac{1}{2}k^2}{4 k^2 - k} < \frac{2 -\frac{1}{2}}{4},
\end{equation*}
which indeed holds for $k = 1$, but fails for all integers $k > 1$.
\end{exm}
\subsubsection{Proof of Theorem \ref{BstabThm}}\label{BstabProof}
Firstly, note that the special Lagrangian condition for the shifted section $\tilde{\cL}$ coincides with the dHYM equation for the the mirror line bundle $L^{k}$ with respect to $k[\omega_0]$. This is solvable iff it is solvable for $L^{\vee}$ with respect to $[\omega_0]$. Since $L^{\vee}$ is ample by assumption, by Proposition \ref{BstabProp}, this holds if the object $L^k[1]$ is Bridgeland stable with respect to the stability condition $\sigma_{k[\omega_0]}$, i.e. if $\tilde{\cL}$ is Bridgeland stable in the sense of Definition \ref{BstabDef}. The converse also holds conditionally on the Arcara-Miles conjecture \ref{AMConj}, by Proposition \ref{BstabProp}, and so for $X = \Bl_{p} \PP^2$ or $X = \Bl_{p, q} \PP^2$.
\subsection{The unstable case}\label{UnstableSec}
In this Section we prove Theorem \ref{UnstableThm}. The proof is an application of a result on weak solutions of the dHYM equation of surfaces due to Datar-Mete-Song \cite{DatarSong_slopes}, which we briefly recall.

Let $(X, [\omega_0], [\alpha_0])$ denote a surface with K\"ahler class and a $(1,1)$-class $[\alpha_0]$. The reference \cite{DatarSong_slopes}, Section 1.3 introduces a \emph{minimal angle} $\theta_{min} \in (0,\pi)$ determined by
\begin{equation*}
\cot \theta_{min} := \sup_D \left\{\frac{([\alpha_0] - D)^2 - [\omega_0]^2}{2 ([\alpha_0] -D)\cdot [\omega_0]},\, ([\alpha_0] - D) \cdot [\omega_0] > 0 \right\}, 
\end{equation*}
where the supremum is taken over effective $\R$-divisors. It is shown that the supremum is in fact achieved by some effective $\R$-divisor $D$, so we have
\begin{equation*}
\cot \theta_{min} = \frac{([\alpha_0] - D)^2 -  [\omega_0]^2}{2 ([\alpha_0] -D)\cdot [\omega_0]}. 
\end{equation*}
The mimimal angle $\theta_{min}$ should be compared to the topological angle $\varphi = \varphi(X, c_1(L), [\omega_0])$ appearing in the Nakai-Moishezon criterion \eqref{dHYMPosIntro}. One shows the inequality
\begin{equation*}
\cot \theta_{min} \geq \cot \varphi(X, c_1(L), [\omega_0]),
\end{equation*} 
such that equality holds if and only if $(X, [\omega_0], c_1(L))$ satisfies the \emph{dHYM semipositivity condition}
\begin{equation}\label{dHYMSemiPos}
\Rea \left(c_1(L) + \ii [\omega_0]\right)^k \cdot \gamma^{m-k} \cdot Z \geq \cot(\varphi)\Imm \left(c_1(L) + \ii [\omega_0]\right)^k \cdot \gamma^{m-k} \cdot Z
\end{equation}
where $Z$ ranges though $m$-dimensional analytic subvarieties, $\gamma$ ranges through K\"ahler classes, and $k = 1,\ldots, m$.  
\begin{thm}[\cite{DatarSong_slopes}, Theorem 1.10]\label{MinSlopesThm} Suppose $[\alpha_0]$, $[\omega_0]$ satisfy $[\alpha_0] \cdot [\omega_0] > 0$. Then, for any K\"ahler form $\omega \in [\omega_0]$, there exists a unique closed current $\cT \in [\alpha_0]$ such that
\begin{align*}
\Rea \langle (\cT  + \ii \omega)^2\rangle = \cot( \theta_{min}) \Imm \langle (\cT + \ii \omega)^2\rangle,\,\cT \geq \cot( \theta_{min})\omega.
\end{align*}
The equation holds in the sense of measures, where $\langle - \rangle$ denotes the non-pluripolar part. 

Moreover we have 
\begin{equation*}
\cot \theta_{min} \geq \cot \varphi(X, c_1(L), [\omega_0]),
\end{equation*}
with equality iff $(X, [\omega_0], L)$ satisfies the dHYM semipositivity condition \eqref{dHYMSemiPos}.
\end{thm}
\subsubsection{Proof of Theorem \ref{UnstableThm}}\label{UnstableProofSec}
Suppose the SYZ Lagrangian $\tilde{\cL}$ is Bridgeland unstable. By definition, this means that the mirror shifted line bundle $L^k[1]$ is unstable (i.e. not semistable) with respect to $\sigma_{k[\omega_0]}$. Under our assumptions, $L^k[1]$ is an element of $\Coh^{\sharp}(X)$, so the instability of $L^k[1]$ implies the instability of $L^k$. As in the proof of Proposition \ref{BstabProp}, assuming the Arcara-Miles conjecture \ref{AMConj}, this means that there exists a curve $C\subset X$, with $C^2<0$, such that
\begin{equation*}
C \cdot \ch_1(L^{ k}) - \frac{1}{2} C^2 < \frac{\ch_2( L^{ k}) -\frac{1}{2}[k\omega_0]^2}{\ch_1(L^{ k})\cdot[k\omega_0]}(C\cdot [k\omega_0]).
\end{equation*}
Then we must have
\begin{equation*}
C \cdot \ch_1(L) < \frac{\ch_2(L) - \frac{1}{2}[\omega_0]^2}{\ch_1(L)\cdot [\omega_0]} (C\cdot [\omega_0]),
\end{equation*}
which implies that the dHYM semipositivity condition \eqref{dHYMSemiPos} cannot hold.

Thus, according to Theorem \ref{MinSlopesThm}, there exists a nontrivial effective $\R$-divisor $D$, such that 
\begin{equation*}
\cot\theta_{min} = \frac{(-c_1(L) - D)^2 - [\omega_0]^2}{2(-c_1(L) - D) \cdot [\omega_0]},\,(-c_1(L) - D) \cdot [\omega_0]>0
\end{equation*} 
with 
\begin{equation*}
\cot\theta_{min} > \cot \varphi(X, -c_1(L), [\omega_0]). 
\end{equation*} 
Note that since $D$ attains the supremum in the definition of $\cot\theta_{min}$, we must have
\begin{equation*}
\cot\theta_{min} = \cot \varphi(X, -L -D , [\omega_0]).
\end{equation*}
Then, it follows from Theorem \ref{MinSlopesThm} again that there exists a unique weak solution of the dHYM equation \eqref{dHYMIntro}, \emph{with the correct topological angle $e^{\ii\phi}$}, i.e. a unique closed current $\cT_1 \in c_1(-L -D )$ which satisfies the equation 
\begin{equation}\label{WeakdHYM}
\Rea \langle (\cT_1 + \ii \omega_0)^2\rangle = \cot(\varphi(X, -L -D , [\omega_0]) \Imm \langle (\cT_1 + \ii \omega_0)^2\rangle
\end{equation}   
in the sense of measures, where $\langle - \rangle$ denotes the non-pluripolar part, and such that  
\begin{equation*}
\cT_1 - \cot(\varphi(X, -L  -D , [\omega_0]) \omega_0 \geq 0.
\end{equation*}
Suppose now that $D$ is in fact a $\Q$-divisor. Then, possibly by taking $k$ even larger, we can assume that $-k L -k D$ is a genuine line bundle. The morphism in $D^b(X)$
\begin{equation*}
L^k[1] \to L^k(kD)[1]   
\end{equation*}
induces a morphism in $\FS(\cY_{q_k}, W(k \omega_0))$, 
\begin{equation*}
\tilde{\cL} \to \tilde{\cL}_D, 
\end{equation*}
where we set $\tilde{\cL}_D := \cL_D[1]$ for $\cL_D$ a SYZ transform of $L^k(kD)$ with respect to $k\omega_0$. 

Since $\cT_1$ is the curvature of a singular Hermitian metric on $-L -D $, the equality of measures \eqref{WeakdHYM}, which is invariant under rescaling $\cT_1$ and $\omega_0$ by $k > 0$, says that $\tilde{\cL}_D$ is Hamiltonian isotopic, in a weak sense, to a (shifted) Lagrangian current section which is a weak solution of the special Lagrangian equation with phase angle $\theta_{min}$.
\begin{rmk} Note that, by Theorem \ref{MinSlopesThm}, the class $c_1(-L -D )$ satisfies the dHYM semipositivity condition \eqref{dHYMSemiPos} with respect to $[\omega_0]$. This implies that the Nakai-Moishezon criterion \eqref{dHYMPosIntro} holds weakly, i.e. for all proper irreducible subvarieties $V \subset X$ we have
\begin{equation}\label{NumericaldHYMSemiPos} 
\int_V \Rea(\ii \omega_0 + c_1(-L-D))^{\dim V} - \cot(\varphi)\Imm(\ii \omega_0 +  c_1(-L-D))^{\dim V} \geq 0.
\end{equation} 
As in the proof of Proposition \ref{BstabProp}, the numerical semipositivity condition \eqref{NumericaldHYMSemiPos} implies that $ L^k(kD)[1]$ is in fact stable with respect to $\sigma_{k[\omega_0]}$. However, in general, this does not automatically show that $\tilde{\cL}_D$ is Hamiltonian isotopic to a \emph{smooth} special Lagrangian (as in Theorem \ref{BstabThm}). For this, $[\omega_0]$ would need to be generic with respect to $c_1(L^k(kD))$, but the divisor $D$ in turn depends on $[\omega_0]$.   
\end{rmk}
\begin{exm}\label{SupportExample} For the sake of this example we change our sign convention to fit with \cite{DatarSong_slopes}, Theorem 1.13. Consider $X = \Bl_p \PP^2$ endowed with the K\"ahler class $[\omega_0] = b [H] - [E]$ and the line bundle $L' = p H - q E$ (with $c_1(L') \cdot [\omega_0] > 0$). If $q < \cot(\varphi(X, L', [\omega_0]))$, then $L'$ does not support a smooth dHYM solution, we have $\cot(\theta_{min}) > \cot(\varphi(X, L', [\omega_0]))$, and there is a unique closed current $\cT \in c_1(L')$ such that
\begin{equation*} 
\Rea \langle (\cT + \ii \omega_0)^2\rangle = \cot(\theta_{min}) \Imm \langle (\cT + \ii \omega_0)^2\rangle,\,\cT \geq \cot(\theta_{min}) [\omega_0].
\end{equation*}   
According to \cite{DatarSong_slopes} Theorem 1.13 (3), we have in this case
\begin{equation*}
\cT = \chi + (\cot(\theta_{min}) - q)[E],
\end{equation*}
where $\chi$ is a $(1,1)$-form with bounded local potentials, solving the equation
\begin{equation*} 
\Rea (\chi + \ii \omega_0)^2 = \cot(\theta_{min}) \Imm (\chi + \ii \omega_0)^2 
\end{equation*}    
in the sense of Bedford-Taylor on $X\setminus E$. 

Moreover, we have 
\begin{equation*}
\cot(\theta_{min}) = \cot \varphi(X, L'(-D), [\omega_0])
\end{equation*}
for some effective $\R$-divisor given by $D = t E$, for some $t > 0$, realising the supremum
\begin{equation*}
\sup_t \left\{\frac{(c_1(L') - t E)^2 - [\omega_0]^2}{2(c_1(L') - t E)\cdot[\omega_0]},\,(c_1(L') - t E)\cdot[\omega_0]>0\right\}.
\end{equation*}
It follows from \cite{DatarSong_slopes} Theorem 1.13 (2) that the unique closed current $\cT_1 \in c_1(L'(-D))$ such that
\begin{equation*} 
\Rea \langle (\cT_1 + \ii \omega_0)^2\rangle = \cot \varphi(X, L'(-D), [\omega_0]) \Imm \langle (\cT_1 + \ii \omega_0)^2\rangle,\,\cT_1 \geq \cot \varphi(X, L'(-D), [\omega_0]) [\omega_0],
\end{equation*}    
is actually more regular. We have in fact
$\cT_1 = \chi_1$,
where $\chi_1$ is a $(1,1)$-form with bounded local potentials, satisfying the equation
\begin{equation*} 
\Rea (\chi_1 + \ii \omega_0)^2 = \cot \varphi(X, L'(-D), [\omega_0]) \Imm (\chi_1 + \ii \omega_0)^2 
\end{equation*}    
in the sense of Bedford-Taylor on all $X$.
\end{exm}
\section{Lower phase: proof of Proposition \ref{LowerPhaseProp}}\label{LowerPhaseSec}
Recall Proposition \ref{LowerPhaseProp} is concerned with the specific example $X = \Bl_p \PP^n$, away from the supercritical phase. According to \cite{JacobSheu}, Theorem 2, in this case, there exists a unique $\hat{\theta} \in \R$, satisfying 
\begin{equation*}
\hat{\theta} = \arg\left(-\int_X e^{-\ii \omega_0}\ch(L^{\vee})\right) \mod 2\pi, 
\end{equation*}
such that the dHYM equation is solvable on $L^{\vee}$ iff the constant Lagrangian phase equation 
\begin{equation*} 
\vartheta(\omega^{-1} \alpha) = \hat{\theta}  
\end{equation*} 
is solvable for $\alpha \in [\alpha_0] = c_1(L^{\vee})$. Moreover, the latter equation is solvable iff for all toric divisors $V \subset X$ we have  
\begin{align}\label{DualJacobSheuIneq}
\hat{\theta} - \frac{\pi}{2}  < \arg\left(-\int_V e^{-\ii \omega_0}\ch(L^{\vee})\right) < \hat{\theta} + \frac{\pi}{2}.
\end{align}
As in Section \ref{PhaseIneqSec}, we note
\begin{align*} 
&\arg\left(-\int_X e^{-\ii \omega_0}\ch(L^{\vee})\right) = \arg\left(-(-1)^n\int_X e^{\ii \omega_0}\ch(L)\right),\\
&\arg\left(-\int_V e^{-\ii \omega_0}\ch(L^{\vee})\right)= \arg\left(-(-1)^{\dim V}\int_V e^{\ii \omega_0}\ch(L)\right).
\end{align*}
So, by complex conjugation, we have
\begin{equation*}
\hat{\theta} = -\arg\left(-(-1)^n\int_X e^{-\ii \omega_0}\ch(L)\right) \mod 2\pi, 
\end{equation*}
and \eqref{DualJacobSheuIneq} can be written as
\begin{align*} 
\hat{\theta} - \frac{\pi}{2}  < -\arg\left( (-1)^{n}\int_V e^{-\ii \omega_0}\ch(L)\right) < \hat{\theta} + \frac{\pi}{2}.
\end{align*}
Then, by \eqref{AsympX}, we have
\begin{equation*} 
\hat{\theta} = -\arg\left(-(-2\pi \ii)^{-n} \int_{\Gamma(L^{\otimes k})} e^{-W(k \omega_0)/z} \Omega_0\right) + O(k^{-1}) \mod 2\pi,
\end{equation*}
and, using \eqref{AsympV}, the inequality \eqref{DualJacobSheuIneq} is equivalent to
\begin{align*}
\hat{\theta} - \frac{\pi}{2} < -\arg\left((-2\pi \ii)^{-n} \int_{\Gamma(\cS(k L, k_V V))} e^{-W(k \omega_0)/z} \Omega_0\right) < \hat{\theta} + \frac{\pi}{2}
\end{align*}
for $k \gg k_V \gg 1$.
\section{Higher rank}
This Section contains the proof of Theorem \ref{HighRankFanoThm}. We consider a rank $2$ vector bundle $E \to X$ on a toric weak del Pezzo surface $X$. Recall that the positivity (i.e. subsolution) and stability conditions \eqref{HighRankPositivity}, \eqref{HighRankStability}, appearing in Theorem \ref{KellerScarpaThm}, imply the phase inequalities \eqref{HighRankPhaseInequ}. 
\subsection{Proof of Theorem \ref{HighRankFanoThm}}\label{HighRankProofSec}
As for Theorem \ref{FanoThm}, the proof involves a parameter $z > 0$, which can be finally specialised to $z = 1$. By assumption, for all $k > 0$, $E_k$ is a nontrivial extension of line bundles, 
\begin{equation*}
0 \to L^{k}_1 \to E_k \to L^{k}_2 \to 0,
\end{equation*}
so in particular we have
\begin{equation*}
\ch(E_k) = \ch(L^k_1 \oplus L^k_2) = \ch(L^k_1) + \ch(L^k_2).
\end{equation*}
We write $E := E_1$. 
\subsubsection{Period of $E$}
Let us consider the integral
\begin{equation}\label{Rank2ScaledGammaIntegral} 
\int_X \left(z^{c_1(X)} z^{\frac{\deg}{2}} e^{2\pi k [\omega_0 - \ii \beta_0]/z} \sum_{d \in A^{\vee}_X \cap \kk_{\Z}} q^{k d} \prod^m_{i = 1}\frac{\prod^0_{j = -\infty}( D_i - j z)}{\prod^{D_i \cdot d}_{j = -\infty} (D_i - j z)}\right)\cup \Gcl_X (2\pi \ii)^{\frac{\deg}{2}} \ch(E_k). 
\end{equation} 
Since $\ch(E_k) = \ch(L^k_1) + \ch(L^k_2)$, \eqref{Rank2ScaledGammaIntegral} can be expanded as
\begin{align*}  
&(2\pi \ii)^2 k^2 \int_X e^{-\ii [\omega_0 - \ii \beta_0]} \ch(L_1) + (2\pi \ii)^2 k^2 \int_X e^{-\ii [\omega_0 - \ii \beta_0]} \ch(L_2) + O(k)\\
&= (2\pi \ii)^2 k^2 \int_X e^{-\ii [\omega_0 - \ii \beta_0]} \ch(L_1\oplus L_2) + O(k)\\
&= (2\pi \ii)^2 k^2 \int_X e^{-\ii [\omega_0 - \ii \beta_0]} \ch(E) + O(k).
\end{align*}
At the same time, by Theorem \ref{GammaThm}, \eqref{Rank2ScaledGammaIntegral} equals the period
\begin{equation*}
\int_{\Gamma(E_k)} e^{-W(k (\omega_0 - \ii \beta_0))/z} \Omega_0.
\end{equation*}
Combining our computations, we find
\begin{align}\label{AsympE}
&\int_X e^{-\ii   [\omega_0 - \ii \beta_0]} \ch(E) = (2\pi \ii)^{-2} k^{-2} \int_{\Gamma(E_k)} e^{-W(k (\omega_0 - \ii \beta_0))/z} \Omega_0 + O(k^{-1}). 
\end{align}
Note that we also have
\begin{align}\label{AsympEk}
&\int_X e^{-\ii k  [\omega_0 - \ii \beta_0]} \ch(E_k) = \left(\frac{1}{2\pi \ii}\right)^2 \int_{\Gamma(E_k)} e^{-W(k (\omega_0 - \ii \beta_0))/z} \Omega_0\,(1 + O(k^{-1})). 
\end{align}
\subsubsection{Restriction to a curve}
Let $V \subset X$ be an irreducible curve. Similarly to \eqref{ScaledGammaIntegralV}, we consider the quantity 
\begin{align}\label{Rank2ScaledGammaIntegralV} 
&\nonumber\int_X\left(z^{c_1(X)} z^{\frac{\deg}{2}} e^{2\pi k [\omega_0 - \ii \beta_0]} \sum_{d \in A^{\vee}_X \cap \kk_{\Z}} q^{k d} \prod^m_{i = 1}\frac{\prod^0_{j = -\infty}( D_i - j z)}{\prod^{D_i \cdot d}_{j = -\infty} (D_i - j z)}\right)\\
&\quad\quad\cup \Gcl_X (2\pi \ii)^{\frac{\deg}{2}} \ch(E_k\otimes \olo(V)^{\otimes k_1}) 
\end{align} 
for $k \gg k_1 \gg 1$. The leading order term in $k$ is the integral \eqref{Rank2ScaledGammaIntegral}; the leading correction is given by
\begin{align*}  
& k_1 \int_X e^{2\pi k   [\omega_0 - \ii \beta_0]} (2\pi \ii)^{\frac{\deg}{2}}\left(\ch(L^k_1\oplus L^k_2) \cup c_1(\olo(V)) \right)\\
& = k_1 \int_X e^{2\pi k   [\omega_0 - \ii \beta_0]} (2\pi \ii)^{\frac{\deg}{2}}\left(\ch(L^k_1) \cup c_1(\olo(V)) \right)\\
&+ k_1 \int_X e^{2\pi k   [\omega_0 - \ii \beta_0]} (2\pi \ii)^{\frac{\deg}{2}}\left(\ch(L^k_2) \cup c_1(\olo(V)) \right)\\
& = (2\pi \ii )^{2} k k_1 \int_X e^{-\ii [\omega_0 - \ii \beta_0]} \ch(L_1) \cup c_1(\olo(V)) \\
& + (2\pi \ii )^{2} k k_1 \int_X e^{-\ii [\omega_0 - \ii \beta_0]} \ch(L_2) \cup c_1(\olo(V)) + O(k, k^0_V)\\
&= (2\pi \ii )^{2} k k_1 \int_V e^{-\ii [\omega_0 - \ii \beta_0]} \ch(E|_{V}) + O(k, k^0_V).
\end{align*} 
At the same time, by Theorem \ref{GammaThm}, the integral \eqref{Rank2ScaledGammaIntegralV} equals the period 
\begin{align*}
\int_{\Gamma(E_k\otimes \olo(V)^{\otimes k_1})} e^{-W(k(\omega_0 - \ii \beta_0))/z} \Omega_0.
\end{align*}
Thus, we find 
\begin{align*}
& \int_V e^{- \ii  [\omega_0 - \ii \beta_0]} \ch(E|_{V})  = (2\pi \ii )^{-2} k^{-1} k^{-1}_1 \int_{\Gamma(E_k\otimes \olo(V)^{\otimes k_1})} e^{-W(k(\omega_0 - \ii \beta_0))/z} \Omega_0\\ 
& = (2\pi \ii )^{-2} k^{-1} k^{-1}_1\int_{\Gamma(\cS_{V})} e^{-W(k(\omega_0 - \ii \beta_0))/z} \Omega_0,  
\end{align*}
where 
\begin{equation*}
\cS_{V} := [E_k \to E_k\otimes \olo(V)^{\otimes k_1}] \in D^b(X),
\end{equation*}
so there is a morphism
\begin{equation*}
\cS_{V}[-1] \to E_k. 
\end{equation*}
Then, by \eqref{AsympE}, under our genericity assumption on $[\omega_0]$, the phase inequality appearing in \eqref{HighRankPhaseInequ},  
\begin{align*}
&\arg\left(\int_X e^{-\ii \omega_0} e^{-\beta_0} \ch(E)\right) -\pi\\
&< \arg\left(\int_V e^{-\ii \omega_0} e^{-\beta_0} \ch(E|_{V})\right) - \pi < \arg\left(\int_X e^{-\ii \omega_0} e^{-\beta_0} \ch(E)\right),
\end{align*}
is equivalent to the phase inequality of periods
\begin{align}\label{HighRankPeriodRestriction}
&\nonumber \arg\left(\int_{\Gamma(E_k)} e^{-W(k (\omega_0 - \ii \beta_0))/z} \Omega_0\right)-\pi\\
&<\arg\left(\int_{\Gamma(\cS_{V})} e^{-W(k(\omega_0 - \ii \beta_0))/z} \Omega_0 \right) - \pi < \arg\left(\int_{\Gamma(E_k)} e^{-W(k (\omega_0 - \ii \beta_0))/z} \Omega_0\right).
\end{align}
\subsubsection{Sub-bundle}
According to \eqref{AsympX}, we have 
\begin{equation*} 
\int_X e^{- \ii   [\omega_0 - \ii \beta_0]} \ch(L_1) = (2\pi \ii)^{-2} k^{-2} \int_{\Gamma(L^{\otimes k}_1)} e^{-W(k (\omega_0 - \ii \beta_0))/z} \Omega_0 + O(k^{-1}).
\end{equation*}
Then, by \eqref{AsympE}, under our genericity assumption on $[\omega_0]$, the phase inequality appearing in \eqref{HighRankPhaseInequ},
\begin{align*}
\arg\left(\int_X e^{-\ii \omega_0} e^{-\beta_0} \ch(E)\right)-\pi<\arg\left(\int_X e^{-\ii \omega_0} e^{-\beta_0} \ch(L_1)\right) < \arg\left(\int_X e^{-\ii \omega_0} e^{-\beta_0} \ch(E)\right),  
\end{align*} 
is equivalent to the phase inequality of periods
\begin{align}\label{HighRankPeriodSubline} 
&\nonumber\int_{\Gamma(E_k)} e^{-W(k (\omega_0 - \ii \beta_0))/z}\Omega_0-\pi\\
&<\arg\left(  \int_{\Gamma(L^{\otimes k}_1)} e^{-W(k(\omega_0 - \ii \beta_0))/z} \Omega_0 \right) < \arg\left(\int_{\Gamma(E_k)} e^{-W(k (\omega_0 - \ii \beta_0))/z} \Omega_0\right). 
\end{align}
\subsection{An example}\label{KellerScarpaExm}
This construction is taken from \cite{KellerScarpa} (Example 3.5). Let $X = \Bl_p \PP^2$. Write $H$ for the class of a line in $\PP^2$, $E'$ for the exceptional divisor, and choose
\begin{equation*}
L_1 = r(q H - p E'),\,L_2 = \olo_X
\end{equation*}
for positive integers $p, q, r$, with $p > q$. Then one can check for all $r \gg 0$ there is a nontrivial extension
\begin{equation*}
0 \to L_1 \to E_1 \to \olo_X \to 0,
\end{equation*}
and so the same holds for 
\begin{equation*}
0 \to L^{\otimes k}_1 \to E_k \to \olo_X \to 0,
\end{equation*}
for $k>0$. Thus, Theorem \ref{HighRankFanoThm} and Corollary \ref{HighRankFSCor} are applicable to this example. According to \cite{KellerScarpa}, if we choose 
\begin{equation*}
[\omega_0] = p H - q E',
\end{equation*}
then there exists a chamber in the space of $B$-fields, containing certain $B$-fields proportional to $[\omega_0]$, such that the conditions \eqref{HighRankPositivity}, \eqref{HighRankStability} hold. It is not yet known, in general, if the dHYM equation \eqref{HigherRankdHYM} is solvable for these $B$-fields (although it is known for some special values of $p, q$, and some $B$-fields, see \cite{KellerScarpa} Example 4.15). 
\subsection{Proof of Proposition \ref{HighRankBInstProp}}\label{HighRankBInstPropProof}
We follow the discussion in Section \ref{BSec}. It is standard to allow a nontrivial $B$-field $\beta_0$, by replacing the Chern character $\ch(E)$ with the twisted Chern character 
\begin{equation*}
\ch^{[\beta_0]}(E) := e^{-[\beta_0]} \ch(E)
\end{equation*}
both in the Mumford-Takemoto slope and in the central charge. We denote the resulting stability condition by
\begin{equation*}
\sigma_{[\omega_0], [\beta_0]} := \left(\Coh^{\sharp}_{[\beta_0]}(X), Z_{[\omega_0], [\beta_0]}\right).
\end{equation*}
Fix $k > 0$. Then, under the assumption
\begin{equation*}
\mu_{k[\omega_0], k[\beta_0]}(L^{\otimes k}_i) \leq 0,\,i = 1, 2,
\end{equation*}
the bundle $E_k$ is an extension of sheaves with nonpositive $k[\beta_0]$-twisted slopes and so we have $H^{-1}(E_k[1]) = E_k \in \Coh^{\leq 0}_{[\beta_0]}(X)$, that is, $E_k[1]$ defines an object of the heart $\Coh^{\sharp}_{k[\beta_0]}(X)$.  

Similarly we have $\mu_{k[\omega_0], k[\beta_0]}(L^{\otimes k}_1)\leq 0$ and so $H^{-1}(L^{\otimes k}_1[1]) = L^{\otimes k}_1 \in \Coh^{\leq 0}_{k[\beta_0]}(X)$, giving $L^{\otimes k}_1[1] \in \Coh^{\sharp}_{[\beta_0]}(X)$. 

On the other hand we note that the only nonvanishing cohomology sheaf $H^0(\cS_{V})$ of $\cS_{V} = [E_k \to E_k\otimes \olo(V)^{\otimes k_1}]$ is the one-dimensional torsion sheaf $E_k \otimes \olo_{k_1 V}$, lying in $\Coh^{>0}_{k[\beta_0]}(X)$ by definition, so we also have $\cS_V \in \Coh^{\sharp}_{k[\beta_0]}(X)$.  

By the proof of Theorem \ref{HighRankFanoThm}, for $k \gg k_1 \gg 1$, the phase inequality for periods \eqref{HighRankPeriodRestriction}
is equivalent to
\begin{equation*} 
\arg\left(Z_{k[\omega_0], k[\beta_0] }(E_k)\right) < \arg\left(Z_{k[\omega_0], k[\beta_0] }(\cS_V)\right)  < \arg\left(Z_{k[\omega_0], k[\beta_0] }(E_k)\right) + \pi.
\end{equation*}
Similarly, \eqref{HighRankPeriodSubline} is equivalent to
\begin{equation*} 
\arg\left(Z_{k[\omega_0], k[\beta_0]}(E_k)\right) - \pi < \arg\left(Z_{k[\omega_0],k[\beta_0]}(L^{\otimes k}_1) \right) < \arg\left(Z_{k[\omega_0],k[\beta_0]}(E_k)\right).
\end{equation*}  
Moreover, the quantities 
\begin{equation*}
Z_{k[\omega_0],k[\beta_0]}(E_k),\,Z_{k[\omega_0],k[\beta_0]}(\cS_V),\,Z_{k[\omega_0],k[\beta_0]}(L^{\otimes k}_1)
\end{equation*}
are central charges of objects in the heart $\Coh^{\sharp}_{k[\beta_0]}(X)$ and so lie in a fixed semi-open half-plane. Therefore, the inequalities \eqref{HighRankPeriodRestriction} and \eqref{HighRankPeriodSubline} are in fact equivalent to 
\begin{equation*} 
\arg\left(Z_{k[\omega_0],k[\beta_0]}(\cS_V)\right) < \arg\left(Z_{k[\omega_0],k[\beta_0]}(E_k[1])\right), 
\end{equation*} 
respectively
\begin{equation*} 
\arg\left(Z_{k[\omega_0],k[\beta_0]}(L^k_1[1]) \right) < \arg\left(Z_{k[\omega_0],k[\beta_0]}(E_k[1])\right),
\end{equation*} 
and so are implied by the Bridegland stability of $E_{k}[1]$ with respect to $\sigma_{k [\omega_0], k[\beta_0]}$, by using the morphisms 
\begin{equation*}
\cS_V \to E_k[1],\,L^k_1[1] \to E_k[1]
\end{equation*}
in $\Coh^{\sharp}_{k[\beta_0]}(X)$.
\begin{rmk} The central charge $Z_{k[\omega_0],k[\beta_0]}(E_k[1])$ is natural for the corresponding Lagrangian $\cL$, by \eqref{AsympEk}.
\end{rmk}
\section{General toric manifolds}
This Section proves Theorem \ref{GeneralThm}. Suppose $X$ is a projective toric manifold. We denote by $V \subset X$ a proper, irreducible, toric subvariety.
\subsection{Proof of Theorem \ref{GeneralThm}}\label{GeneralToricProofSec}
Using the mirror map with respect to $\omega_0$, by the properties recalled in Section \ref{MirrorSec}, we have a equalities
\begin{align*}
& \int_{X} e^{-\ii \omega_0} \ch(L) = P_{\beta}\left(\Theta^{-1}_{\omega_0}(e^{-\ii [\omega_0]}), \Theta^{-1}_{\omega_0}(\ch(L))\right),\\
& \int_{V} e^{-\ii \omega_0} \ch(L) = P_{\beta}\left(\Theta^{-1}_{\omega_0}(e^{-\ii [\omega_0]}), \Theta^{-1}_{\omega_0}(\ch(L) \cup \pd(V))\right), 
\end{align*}
involving the higher residue pairing $P$. The quantities
\begin{equation*}
\int_{X} e^{-\ii [\omega_0]} \ch(L),\,\int_{V} e^{-\ii [\omega_0]} \ch(L)
\end{equation*}
clearly do not depend on the formal variable $z$ appearing in the higher residue pairing. Thus, we must have
\begin{align*}
& \int_{X} e^{-\ii [\omega_0]} \ch(L) =   K^{(0)}_{[\omega_0]}\left(\Theta^{-1}_{[\omega_0]}(e^{-\ii [\omega_0]}), \Theta^{-1}_{[\omega_0]}(\ch(L))\right),\\
& \int_{V} e^{-\ii [\omega_0]} \ch(L) =   K^{(0)}_{[\omega_0]}\left(\Theta^{-1}_{[\omega_0]}(e^{-\ii [\omega_0]}), \Theta^{-1}_{[\omega_0]}(\ch(L) \cup \pd(V))\right),
\end{align*}
where $K^{(0)}_{[\omega_0]}$ denotes the specialisation of the higher residue pairing $P_{[\omega_0]}$ at $z = 0$, namely, the classical Grothendieck residue pairing. 

It follows that, with our assumptions, \eqref{PhaseIneq} is equivalent to the phase inequality
\begin{align*}
&\arg\left((-1)^{\codim V} K^{(0)}_{[\omega_0]}\left(\Theta^{-1}_{[\omega_0]}(e^{-\ii [\omega_0]}), \Theta^{-1}_{[\omega_0]}(\ch(L) \cup \pd(V))\right)\right)\\ 
& < \arg K^{(0)}_{[\omega_0]}\left(\Theta^{-1}_{[\omega_0]}(e^{-\ii [\omega_0]}), \Theta^{-1}_{[\omega_0]}(\ch(L))\right). 
\end{align*}

Now we consider the large volume limit. Namely, by the same argument, we observe that \eqref{PhaseIneq}  holds in our case iff we have
\begin{align*}
&\arg\left((-1)^{\codim V} K^{(0)}_{k[\omega_0]}\left(\Theta^{-1}_{k[\omega_0]}(e^{-\ii [\omega_0]}), \Theta^{-1}_{k[\omega_0]}(\ch(L) \cup \pd(V))\right)\right)\\
& < \arg K^{(0)}_{k[\omega_0]}\left(\Theta^{-1}_{k[\omega_0]}(e^{-\ii [\omega_0]}), \Theta^{-1}_{k[\omega_0]}(\ch(L))\right) 
\end{align*}
for all sufficiently large $k > 0$. Since we have
\begin{align*}
&K^{(0)}_{k[\omega_0]}\left(\Theta^{-1}_{k[\omega_0]}(e^{-\ii [\omega_0]}), \Theta^{-1}_{k[\omega_0]}(\ch(L) \cup \pd(V))\right) \\
&= K^{(0)}_{k[\omega_0]}\left(\Theta^{-1}_{k[\omega_0]}(e^{-\ii [\omega_0]}), \Theta^{-1}_{k[\omega_0]}(\ch(L)) \Theta^{-1}_{k[\omega_0]}(\pd(V))\right) + O(k^{-1})
\end{align*}
we can express our condition as
\begin{align*}
&\arg\left((-1)^{\codim V} K^{(0)}_{k[\omega_0]}\left(\Theta^{-1}_{k[\omega_0]}(e^{-\ii [\omega_0]}), \Theta^{-1}_{k[\omega_0]}(\ch(L)) \Theta^{-1}_{k[\omega_0]}(\pd(V))\right)\right)\\
&< \arg K^{(0)}_{k[\omega_0]}\left(\Theta^{-1}_{k\beta}(e^{-\ii [\omega_0]}), \Theta^{-1}_{k[\omega_0]}(\ch(L))\right), 
\end{align*}
for $k \gg 1$, as long as $[\omega_0]$ is \emph{generic} in the sense that it does not lie in the union of finitely many proper analytic subvarieties of $H^{1, 1}(X, \R)$ determined by $c_1(L)$ (the union ranging through all toric irreducible toric subvarieties $V$).

Using the expression for the Grothendieck residue in terms of critical points, we find 
\begin{align*}
&K^{(0)}_{k[\omega_0]}\left(\Theta^{-1}_{k[\omega_0]}(e^{-\ii [\omega_0]}), \Theta^{-1}_{k[\omega_0]}(\ch(L)) \Theta^{-1}_{k[\omega_0]}(\pd( V ))\right)\\
& = \sum_{p \in \Crit(W_{k[\omega_0]})} \frac{\Theta^{-1}_{k[\omega_0]}(e^{-\ii [\omega_0]}) \Theta^{-1}_{k[\omega_0]}(\ch(L))}{\prod_i x^2_i \det \nabla^2 W_{k[\omega_0]}}\big|_p \, \Theta^{-1}_{k[\omega_0]}(\pd( V ))\big|_p\frac{\Omega_{k[\omega_0]}}{\Omega_0}\big|_p,
\end{align*}
and similarly
\begin{align*}
&K^{(0)}_{k[\omega_0]}\left(\Theta^{-1}_{k[\omega_0]}(e^{-\ii [\omega_0]}), \Theta^{-1}_{k[\omega_0]}(\ch(L))\right) = \sum_{p \in \Crit(W_{k[\omega_0]})} \frac{\Theta^{-1}_{k[\omega_0]}(e^{-\ii [\omega_0]}) \Theta^{-1}_{k[\omega_0]}(\ch(L))}{\prod_i x^2_i \det \nabla^2 W_{k[\omega_0]}}\big|_p \, \frac{\Omega_{k[\omega_0]}}{\Omega_0}\big|_p.
\end{align*}

Applying \cite{CoatesCortiIritani_hodge}, Lemma 6.2, if we choose an appropriate normalisation of $W_{k[\omega_0]}$, then, for fixed $V$, there exist a \emph{proper} subset 
\begin{equation*}
\emptyset \neq \Lambda_V \subsetneq \Crit(W_{k[\omega_0]}) 
\end{equation*}
and constants $c_V(p) \in \Z$ such that 
\begin{equation*}
\Theta^{-1}_{k[\omega_0]}(\pd( V))\big|_p = \left\{ 
\begin{matrix} c_V(p) + O(k^{-1}),\,p \in \Lambda_V \\
O(k^{-1}),\, p \notin \Lambda_V.
\end{matrix}\right.
\end{equation*}
So, for fixed $V$, we have 
\begin{align*}
&K^{(0)}_{k[\omega_0]}\left(\Theta^{-1}_{k[\omega_0]}(e^{-\ii [\omega_0]}), \Theta^{-1}_{k[\omega_0]}(\ch(L)) \Theta^{-1}_{k[\omega_0]}(\pd(V))\right)\\
& = \sum_{p \in \Lambda_V} c_{V}(p)\frac{\Theta^{-1}_{k[\omega_0]}(\ch(L))}{\prod_i x^2_i \det \nabla^2 W_{k[\omega_0]}}\big|_p \,\Theta^{-1}_{k[\omega_0]}(e^{-\ii [\omega_0]})\big|_p\frac{\Omega_{k[\omega_0]}}{\Omega_0}\big|_p + O(k^{-1}).
\end{align*}
Setting
\begin{align*}
\Psi_p( \xi ) := \frac{\xi|_p}{\prod_i x^2_i \det \nabla^2 W_{k[\omega_0]}}\big|_p \,\Theta^{-1}_{k[\omega_0]}(e^{-\ii [\omega_0]})\big|_p\frac{\Omega_{k[\omega_0]}}{\Omega_0}\big|_p
\end{align*}
gives a linear function 
$\Psi_p \in \Hom(\GM(W(k\omega_0))|_{z = 0}, \C)$
which is then induced by integration along a complex cycle $\Gamma_p$, with
$[\Gamma_p] \in H_n(\cY_{q_k}, \{\Rea(W(k\omega_0)/z) \gg 0\}; \Z) \otimes \C$, 
that is, we have, in the sense of \cite{CoatesCortiIritani_hodge}, Section 6.2,
\begin{equation*}
\Psi_p(\xi) = \int_{\Gamma_p} e^{-W(k\omega_0)/z}\Omega^{(k)},\,z \sim 0,
\end{equation*}
where 
$\Omega^{(k)} := \frac{\Omega_{k[\omega_0]}}{\Omega_0} \Theta^{-1}_{k[\omega_0]}(e^{-\ii [\omega_0]}) \in \GM(W(k\omega_0))$.
So, setting 
\begin{align*}
& \Gamma_L := \sum_{p\in\Crit(W_{k[\omega_0]})} \Theta^{-1}_{k[\omega_0]}(\ch(L))|_p\Gamma_p,\,\Gamma_{L, V} := \sum_{p \in \Lambda_V} c_V(p) \Theta^{-1}_{k[\omega_0]}(\ch(L))|_p\Gamma_p, 
\end{align*}
we find
\begin{align*}
K^{(0)}_{k[\omega_0]}\left(\Theta^{-1}_{k[\omega_0]}(e^{-\ii [\omega_0]}), \Theta^{-1}_{k[\omega_0]}(\ch(L)) \Theta^{-1}_{k[\omega_0]}(\pd( V ))\right) = \int_{\Gamma_{L,V}} e^{-W(k\omega_0)/z} \Omega^{(k)},
\end{align*}
and similarly
\begin{align*}
K^{(0)}_{k[\omega_0]}\left(\Theta^{-1}_{k[\omega_0]}(e^{-\ii [\omega_0]}), \Theta^{-1}_{k[\omega_0]}(\ch(L))\right) = \int_{\Gamma_L} e^{-W(k\omega_0)/z} \Omega^{(k)},
\end{align*}
as $z \to 0^+$.

\addcontentsline{toc}{section}{References}
 
\bibliographystyle{abbrv}
 \bibliography{biblio_dHYM}

\begin{thebibliography}{10}

\bibitem{Abouzaid_toricHMS}
M.~Abouzaid.
\newblock Morse homology, tropical geometry, and homological mirror symmetry
  for toric varieties.
\newblock {\em Selecta Math. (N.S.)}, 15(2):189--270, 2009.

\bibitem{Iritani_SYZGamma}
M.~Abouzaid, S.~Ganatra, H.~Iritani, and N.~Sheridan.
\newblock The gamma and {S}trominger-{Y}au-{Z}aslow conjectures: a tropical
  approach to periods.
\newblock {\em Geom. Topol.}, 24(5):2547--2602, 2020.

\bibitem{ArcaraMiles}
D.~Arcara and E.~Miles.
\newblock Bridgeland stability of line bundles on surfaces.
\newblock {\em J. Pure Appl. Algebra}, 220(4):1655--1677, 2016.

\bibitem{DirichletBook}
P.~Aspinwall, T.~Bridgeland, A.~Craw, M.~Douglas, M.~Gross, A.~Kapustin,
  G.~Moore, G.~Segal, B.~Szendr\H{o}i, and P.~Wilson.
\newblock {\em Dirichlet branes and mirror symmetry}, volume~4 of {\em Clay
  Mathematics Monographs}.
\newblock American Mathematical Society, Providence, RI; Clay Mathematics
  Institute, Cambridge, MA, 2009.

\bibitem{Chan_survey}
K.~Chan.
\newblock S{YZ} mirror symmetry for toric varieties.
\newblock In {\em Handbook for mirror symmetry of {C}alabi-{Y}au \& {F}ano
  manifolds}, volume~47 of {\em Adv. Lect. Math. (ALM)}, pages 1--32. Int.
  Press, Somerville, MA, [2020] \copyright 2020.

\bibitem{ChanLeung_SYZ}
K.~Chan and N.~C. Leung.
\newblock Mirror symmetry for toric {F}ano manifolds via {SYZ} transformations.
\newblock {\em Adv. Math.}, 223(3):797--839, 2010.

\bibitem{GaoChen_Jeq_dHYM}
G.~Chen.
\newblock The {J}-equation and the supercritical deformed
  {H}ermitian-{Y}ang-{M}ills equation.
\newblock {\em Invent. Math.}, 225(2):529--602, 2021.

\bibitem{Takahashi_dHYM}
J.~Chu, M.-C. Lee, and R.~Takahashi.
\newblock A {N}akai-{M}oishezon type criterion for supercritical deformed
  {H}ermitian-{Y}ang-{M}ills equation.
\newblock {\em J. Differential Geom.}, 126(2):583--632, 2024.

\bibitem{CoatesCortiIritani_hodge}
T.~Coates, A.~Corti, H.~Iritani, and H.-H. Tseng.
\newblock Hodge-theoretic mirror symmetry for toric stacks.
\newblock {\em J. Differential Geom.}, 114(1):41--115, 2020.

\bibitem{CollinsJacobYau}
T.~Collins, A.~Jacob, and S.-T. Yau.
\newblock $(1,1)$ forms with specified {L}agrangian phase: a priori estimates
  and algebraic obstructions.
\newblock {\em Camb. J. Math. 8 (2020)}, 8(2):407--452, 2020.

\bibitem{Collins_stability}
T.~C. Collins and Y.~Shi.
\newblock Stability and the deformed {H}ermitian-{Y}ang-{M}ills equation.
\newblock In {\em Surveys in differential geometry 2019. {D}ifferential
  geometry, {C}alabi-{Y}au theory, and general relativity. {P}art 2}, volume~24
  of {\em Surv. Differ. Geom.}, pages 1--38. Int. Press, Boston, MA, [2022]
  \copyright 2022.

\bibitem{CollinsSzekelyhidi}
T.~C. Collins and G.~Sz\'ekelyhidi.
\newblock Convergence of the {J}-flow on toric manifolds.
\newblock {\em J. Differential Geom.}, 107(1):47--81, 2017.

\bibitem{CollinsYau_momentmaps_preprint}
T.~C. Collins and S.-T. Yau.
\newblock Moment maps, nonlinear {PDE}, and stability in mirror symmetry.
\newblock \href{https://arxiv.org/abs/1811.04824v2}{arXiv:1811.04824
  [math.DG]}.

\bibitem{CollinsYau_dHYM_momentmap}
T.~C. Collins and S.-T. Yau.
\newblock Moment maps, nonlinear {PDE} and stability in mirror symmetry, {I}:
  geodesics.
\newblock {\em Ann. PDE}, 7(1):Paper No. 11, 73, 2021.

\bibitem{DatarSong_slopes}
V.~Datar, R.~Mete, and J.~Song.
\newblock Minimal slopes and bubbling for complex {H}essian equations.
\newblock \href{https://arxiv.org/abs/2312.03370}{arXiv:2312.03370 [math.DG]}.

\bibitem{DatarPingali_dHYM}
V.~V. Datar and V.~P. Pingali.
\newblock A numerical criterion for generalised {M}onge-{A}mp\`ere equations on
  projective manifolds.
\newblock {\em Geom. Funct. Anal.}, 31(4):767--814, 2021.

\bibitem{Dervan_Zconnections}
R.~Dervan, J.~B. McCarthy, and L.~M. Sektnan.
\newblock {$Z$}-critical connections and {B}ridgeland stability conditions.
\newblock {\em Camb. J. Math.}, 12(2):253--355, 2024.

\bibitem{Fang_charges}
B.~Fang.
\newblock Central charges of {T}-dual branes for toric varieties.
\newblock {\em Trans. Amer. Math. Soc.}, 373(6):3829--3851, 2020.

\bibitem{Zaslow_toricHMS}
B.~Fang, C.-C.~M. Liu, D.~Treumann, and E.~Zaslow.
\newblock T-duality and homological mirror symmetry for toric varieties.
\newblock {\em Adv. Math.}, 229(3):1875--1911, 2012.

\bibitem{Givental_toric}
A.~Givental.
\newblock A mirror theorem for toric complete intersections.
\newblock In {\em Topological field theory, primitive forms and related topics
  ({K}yoto, 1996)}, volume 160 of {\em Progr. Math.}, pages 141--175.
  Birkh\"{a}user Boston, Boston, MA, 1998.

\bibitem{Iritani_gamma}
H.~Iritani.
\newblock An integral structure in quantum cohomology and mirror symmetry for
  toric orbifolds.
\newblock {\em Adv. Math.}, 222(3):1016--1079, 2009.

\bibitem{Iritani_survey}
H.~Iritani.
\newblock Quantum {D}-modules of toric varieties and oscillatory integrals.
\newblock In {\em Handbook for mirror symmetry of {C}alabi-{Y}au \& {F}ano
  manifolds}, volume~47 of {\em Adv. Lect. Math. (ALM)}, pages 131--147. Int.
  Press, Somerville, MA, [2020] \copyright 2020.

\bibitem{JacobSheu}
A.~Jacob and N.~Sheu.
\newblock The deformed {H}ermitian-{Y}ang-{M}ills equation on the blowup of
  {$\Bbb P^n$}.
\newblock {\em Asian J. Math.}, 26(6):847--864, 2022.

\bibitem{JacobYau_special_Lag}
A.~Jacob and S.-T. Yau.
\newblock A special {L}agrangian type equation for holomorphic line bundles.
\newblock {\em Math. Ann.}, 369(1-2):869--898, 2017.

\bibitem{Joyce_ThomasYau}
D.~Joyce.
\newblock Conjectures on {B}ridgeland stability for {F}ukaya categories of
  {C}alabi-{Y}au manifolds, special {L}agrangians, and {L}agrangian mean
  curvature flow.
\newblock {\em EMS Surv. Math. Sci.}, 2(1):1--62, 2015.

\bibitem{KellerScarpa}
J.~Keller and C.~Scarpa.
\newblock {$Z$}-critical equations for holomorphic vector bundles on {K}\"ahler
  surfaces.
\newblock \href{https://arxiv.org/abs/2405.03312}{arXiv:2405.03312 [math.DG]}.

\bibitem{SohaibZak_higherdim}
S.~Khalid and Z.~Sj\"ostr\"om~Dyrefelt.
\newblock {W}all-chamber decompositions for generalized {M}onge-{A}mp\`ere
  equations.
\newblock \href{https://arxiv.org/abs/2412.20089}{arXiv:2412.20089 [math.DG]}.

\bibitem{KuwagakiCohCon}
T.~Kuwagaki.
\newblock The nonequivariant coherent-constructible correspondence for toric
  stacks.
\newblock {\em Duke Math. J.}, 169(11):2125--2197, 2020.

\bibitem{LeungYauZaslow}
N.~C. Leung, S.-T. Yau, and E.~Zaslow.
\newblock From special {L}agrangian to {H}ermitian-{Y}ang-{M}ills via
  {F}ourier-{M}ukai.
\newblock {\em Adv. Theor. Math. Phys.}, 4(6):1319--1341, 2000.

\bibitem{YangLi_ThomasYau}
Y.~Li.
\newblock {T}homas-{Y}au conjecture and holomorphic curves.
\newblock arXiv:2203.01467 [math.SG].

\bibitem{MizunoYoshida}
Y.~Mizuno and T.~Yoshida.
\newblock {B}ridgeland stability of sheaves on del {P}ezzo surfaces of {P}icard
  rank three.
\newblock \href{https://arxiv.org/abs/2502.18894}{arXiv:2502.18894 [math.AG]}.

\bibitem{Pingali_vectorMA}
V.~P. Pingali.
\newblock A vector bundle version of the {M}onge-{A}mp\`ere equation.
\newblock {\em Adv. Math.}, 360:106921, 40, 2020.

\bibitem{SibillaCohCon}
S.~Scherotzke and N.~Sibilla.
\newblock The non-equivariant coherent-constructible correspondence and a
  conjecture of {K}ing.
\newblock {\em Selecta Math. (N.S.)}, 22(1):389--416, 2016.

\bibitem{Song_NakaiMoishezon}
J.~Song.
\newblock {N}akai-{M}oishezon criterions for complex {H}essian equations.
\newblock arXiv:2012.07956 [math.DG].

\bibitem{Takahashi_Jequ}
R.~Takahashi.
\newblock {$J$}-equation on holomorphic vector bundles.
\newblock {\em J. Funct. Anal.}, 286(4):Paper No. 110265, 64, 2024.

\bibitem{Thomas_MomentMirror}
R.~P. Thomas.
\newblock Moment maps, monodromy and mirror manifolds.
\newblock In {\em Symplectic geometry and mirror symmetry ({S}eoul, 2000)},
  pages 467--498. World Sci. Publ., River Edge, NJ, 2001.

\bibitem{ThomasYau}
R.~P. Thomas and S.-T. Yau.
\newblock Special {L}agrangians, stable bundles and mean curvature flow.
\newblock {\em Comm. Anal. Geom.}, 10(5):1075--1113, 2002.

\bibitem{ZhouCohCon}
P.~Zhou.
\newblock Twisted polytope sheaves and coherent-constructible correspondence
  for toric varieties.
\newblock {\em Selecta Math. (N.S.)}, 25(1):Paper No. 1, 23, 2019.

\end{thebibliography}

\noindent SISSA, via Bonomea 265, 34136 Trieste, Italy\\
Institute for Geometry and Physics (IGAP), via Beirut 2, 34151 Trieste, Italy\\
jstoppa@sissa.it    
\end{document}